\documentclass[leqno,12pt]{article}
\usepackage{amsfonts,amsmath,amssymb}
\usepackage{amsthm}
\usepackage{enumerate}
\usepackage{graphics}
\usepackage{epic}
\usepackage{eepic}
\usepackage{float}%
\usepackage[hypertex]{hyperref}%

\usepackage{enumerate}

\numberwithin{equation}{section} %
\numberwithin{table}{section}

  \newtheorem{thm}{Theorem}[section] %

\newtheorem{fact}[thm]{Fact}

  \newtheorem{prop}[thm]{Proposition} %
  \newtheorem{lemma}[thm]{Lemma} %
  \newtheorem{cor}[thm]{Corollary}%
  \newtheorem{df}[thm]{Definition} %
  \newtheorem{exmp}[thm]{Example}

  \newtheorem{remark}[thm]{Remark}%
\providecommand{\bysame}{\makebox[3em]{\hrulefill}\thinspace}
\begin{document}
\title{Classification of finite-multiplicity symmetric pairs
\\
\small{{\it{In honour of Professor Dynkin 
for his 90th birthday}}}
}
\author{Toshiyuki KOBAYASHI
\footnote
{Kavli IPMU (WPI) and Graduate School of Mathematical Sciences,
 the University of Tokyo, 
Meguro-ku, Tokyo, 153-8914, Japan, 
E-mail address: toshi@ms.u-tokyo.ac.jp}
\,\,
 and Toshihiko MATSUKI
\footnote
{
Faculty of Letters, Ryukoku University, Kyoto, 
612-8577, Japan, 
E-mail address: matsuki@let.ryukoku.ac.jp
}
} %
\date{} %
\maketitle %
\noindent
\begin{abstract}
We give a complete classification 
 of the reductive symmetric pairs $(G,H)$
 for which the homogeneous space
 $(G \times H)/\operatorname{diag}H$ is real spherical 
 in the sense
 that a minimal parabolic subgroup
 has an open orbit.  
Combining with a criterion established
 in [T. Kobayashi--T. Oshima, 
Adv. Math. 2013], 
we give a necessary and sufficient condition
 for a reductive symmetric pair $(G,H)$
 such that the multiplicities
 for the branching law
 of the restriction of 
 any admissible smooth representation
 of $G$ to $H$ have finiteness/boundedness property.  
\end{abstract}
\par\noindent
{\bf{MSC 2010;}}\enspace
primary 22E46;
secondary 14M15, 53C35.  

\tableofcontents

\section{Introduction and statement of main results}
\setcounter{subsection}{0}

A complex manifold $X_{\mathbb{C}}$ 
 with action of a complex reductive group $G_{\mathbb{C}}$
 is called {\it{spherical}} 
 if a Borel subgroup of $G_{\mathbb{C}}$
 has an open orbit in $X_{\mathbb{C}}$.  
In the real setting, 
 in search of a good framework for global analysis 
 on homogeneous spaces
 which are broader than the usual 
 ({\it{e.g.}} symmetric spaces),
 we advocated
 in \cite{xtoshi95}
 the importance of an analogous notion
 for real reductive groups $G$ 
 and proposed to call:
\begin{df}
\label{def:realsp}
\rm{
We say a smooth manifold $X$ with $G$-action
 is {\it{real spherical}}
 if a minimal parabolic subgroup $P_G$
 of $G$ has an open orbit in $X$.  
}
\end{df}
In the case
 where $G$ acts transitively on $X$, 
 $P_G$ has finitely many orbits in $X$
 if $X$ is real spherical
 (see \cite[Remark 2.5 (4)]{xtoshitoshima}
 and references therein).

Suppose that $H$ is a closed subgroup
 which is reductive in $G$.  
Let $P_H$ be a minimal parabolic subgroup
 of $H$.  
\begin{df}
[{\cite{xtoshitoshima}}]
\label{def:pp}
\rm{
We say the pair $(G,H)$ satisfies (PP)
 if one of the following four equivalent
 conditions is satisfied.  
\begin{enumerate}
\item[{\rm{(PP1)}}]
$(G \times H)/\operatorname{diag}H$ is real spherical
 as a $(G \times H)$-space.  

\item[{\rm{(PP2)}}]
$G/P_H$ is real spherical as a $G$-space.  
\item[{\rm{(PP3)}}]
$G$ has an open orbit in $G/P_G \times G/P_H$
 via the diagonal action.  
\item[{\rm{(PP4)}}]
There are finitely many $G$-orbits in $G/P_G \times G/P_H$
 via the diagonal action.  
\end{enumerate}
}
\end{df}
The above four equivalent conditions are determined 
 only by the Lie algebras ${\mathfrak {g}}$ and ${\mathfrak {h}}$
 of the Lie groups $G$ and $H$, 
 respectively.  
Therefore we also say
 that the pair $({\mathfrak {g}}, {\mathfrak {h}})$
 of Lie algebras
 satisfies (PP).

A natural question is
 to find all the pairs
 $({\mathfrak {g}}, {\mathfrak {h}})$
 of real reductive Lie algebras
 satisfying (PP)
 when ${\mathfrak {h}}$ is maximal reductive in 
 ${\mathfrak {g}}$.  

We say $(G,H)$ is a {\it{reductive symmetric pair}}
 if $H$ is an open subgroup
 of the fixed point subgroup $G^{\sigma}$
 of some involutive automorphism $\sigma$ of $G$.  
Reductive symmetric pairs $(G,H)$
 give typical examples
 of maximal reductive subalgebras
 ${\mathfrak {h}}$
 in ${\mathfrak {g}}$,
 and provide important setups
 in branching laws
 of the restriction $G \downarrow H$.

The main goal of this paper is to establish a complete 
 classification
 of reductive symmetric pairs 
 $(G,H)$
 having the geometric condition {\rm{(PP)}}.  

\begin{thm}
\label{thm:1.1}
Suppose $(G,H)$ is a reductive symmetric pair.  
Then the following two conditions
 are equivalent:
\begin{enumerate}
\item[{\rm{(i)}}]
$(G,H)$ satisfies {\rm{(PP)}}, 
 namely,
 $(G \times H)/\operatorname{diag}H$
 is real spherical.  

\item[{\rm{(ii)}}]
The pair $({\mathfrak{g}},{\mathfrak{h}})$ of the Lie algebras 
 is isomorphic 
 (up to outer automorphisms)
 to the direct sum of the following pairs:
\begin{enumerate}
\item[{\rm{A)}}]
{\rm{Trivial case:}}
${\mathfrak{g}}={\mathfrak{h}}$.  
\item[{\rm{B)}}]
{\rm{Abelian case:}}
${\mathfrak {g}}={\mathbb{R}}$, 
${\mathfrak {h}}=\{0\}$.  
\item[{\rm{C)}}]
{\rm{Compact case:}}
${\mathfrak {g}}$ is the Lie algebra
 of a compact simple Lie group.  
\item[{\rm{D)}}]
{\rm{Riemannian symmetric pair:}}
${\mathfrak {h}}$ is the Lie algebra 
 of a maximal compact subgroup $K$
 of a non-compact simple Lie group $G$.  
\item[{\rm{E)}}]
{\rm{Split rank one case ($\operatorname{rank}_{{\mathbb{R}}}G=1$):}}
\begin{enumerate}
\item[{\rm{E1)}}]
$({\mathfrak{o}}(p+q,1),
{\mathfrak{o}}(p)+{\mathfrak{o}}(q,1))$
\quad\quad\,\, $(p+q \ge 2)$.
\item[{\rm{E2)}}]
$({\mathfrak{su}}(p+q,1),
{\mathfrak{s}}({\mathfrak {u}}(p)+{\mathfrak{u}}(q,1)))$
\quad
$(p+q \ge 1)$.
\item[{\rm{E3)}}]
$({\mathfrak{sp}}(p+q,1),
{\mathfrak{sp}}(p)+{\mathfrak{sp}}(q,1))$
\quad\,
$(p+q \ge 1)$.
\item[{\rm{E4)}}]
$({\mathfrak{f}}_{4(-20)},
{\mathfrak{o}}(8,1))$.  
\end{enumerate}
\item[{\rm{F)}}]
{\rm{Strong Gelfand pairs
 and their real forms:}}
\begin{enumerate}
\item[{\rm{F1)}}]
$({\mathfrak{sl}}(n+1,{\mathbb{C}}),
{\mathfrak{gl}}(n,{\mathbb{C}}))$
\, $(n\ge 2)$.  
\item[{\rm{F2)}}]
$({\mathfrak{o}}(n+1,{\mathbb{C}}),
{\mathfrak{o}}(n,{\mathbb{C}}))$
\,\,\,\, $(n\ge 2)$.  
\item[{\rm{F3)}}]
$({\mathfrak{sl}}(n+1,{\mathbb{R}}),
{\mathfrak{gl}}(n,{\mathbb{R}}))$ 
\,\,\,$(n\ge 1)$.  
\item[{\rm{F4)}}]
$({\mathfrak{su}}(p+1,q),{\mathfrak{u}}(p,q))$
\,\,\,\,\, $(p+q\ge 1)$.  
\item[{\rm{F5)}}]
$({\mathfrak{o}}(p+1,q),{\mathfrak{o}}(p,q))$ 
\,\,\,\,\,\,\,\,\,\,$(p+q\ge 2)$.  
\end{enumerate}
\item[{\rm{G)}}]
$({\mathfrak{g}}, {\mathfrak{h}})=
 ({\mathfrak{g}}'+{\mathfrak{g}}', \operatorname{diag} {\mathfrak{g}}')$
{\rm{Group case:}}
\begin{enumerate}
\item[{\rm{G1)}}]
${\mathfrak{g}}'$ is the Lie algebra
 of a compact simple Lie group.  
\item[{\rm{G2)}}]
$({\mathfrak{o}}(n,1)+{\mathfrak{o}}(n,1), \operatorname{diag}
{\mathfrak{o}}(n,1))$
\quad
$(n \ge 2)$.  
\end{enumerate}
\item[{\rm{H)}}]
{\rm{Other cases:}}
\begin{enumerate}
\item[{\rm{H1)}}]
$({\mathfrak{o}}(2n, 2),
{\mathfrak{u}}(n,1))$
\hphantom{MMMMMMMMM}
$(n \ge 1)$.
\item[{\rm{H2)}}]
$({\mathfrak{su}}^{\ast}(2n+2),
{\mathfrak{su}}(2)+{\mathfrak{su}}^{\ast}(2n)+\mathbb{R})$ 
\quad
$(n\ge 1)$.  
\item[{\rm{H3)}}]
$({\mathfrak{o}}^{\ast}(2n+2),
{\mathfrak{o}}(2)+{\mathfrak{o}}^{\ast}(2n))$
\hphantom{MMMMn} $(n\ge 1)$.
\item[{\rm{H4)}}]
$({\mathfrak{sp}}(p+1,q),
{\mathfrak{sp}}(p,q)+{\mathfrak{sp}}(1))$.  
\item[{\rm{H5)}}]
$({\mathfrak{e}}_{6(-26)},
{\mathfrak{so}}(9,1)+{\mathbb{R}})$.
\end{enumerate}
\end{enumerate}

\end{enumerate}
\end{thm}

In the above description of the classification,
 we do not intend 
 to write irreducible symmetric pairs
 in an exclusive way.  
Indeed some of the above pairs are  
 isomorphic to each other 
 when ${\mathfrak {g}}$ is of small dimension.  
For instance, 
 (E1) with $(p,q)=(4,1)$ is isomorphic to (H2) with $n=1$, 
 namely, 
\begin{equation*}
({\mathfrak{o}}(5,1), {\mathfrak{o}}(4)+{\mathfrak{o}}(1,1))
\simeq
({\mathfrak{su}}^{\ast}(4), 
{\mathfrak{su}}(2)+{\mathfrak{su}}^{\ast}(2)+{\mathbb{R}}).  
\end{equation*}

\vskip 1pc
\begin{remark}
\label{rem:Dynkin}
{\rm{
It would be interesting
 to give a complete list
 of the pairs $({\mathfrak {g}}, {\mathfrak {h}})$
 of reductive Lie algebras 
 having the property (PP)
 by dropping the assumption
 that $({\mathfrak {g}}, {\mathfrak {h}})$
 is a symmetric pair.  
(Cf. Dynkin \cite{Dynkin}
 for the description of maximal reductive 
 Lie algebras 
 in simple Lie algebras over ${\mathbb{C}}$.)
In view of the classification
 in Theorem \ref{thm:1.1}
 it is plausible
 that there are not many non-symmetric pairs
 $({\mathfrak {g}}, {\mathfrak {h}})$
 satisfying (PP)
 if $H$ is noncompact.  
}}
\end{remark}

Next we also consider another property, 
 to be denoted by (BB), 
 which is stronger than (PP).  
For this, 
suppose further
 that $G$ is an algebraic reductive group 
 and $H$ is a reductive subgroup
 defined algebraically over ${\mathbb{R}}$.  
Let $G_{\mathbb{C}}$ be a complex Lie group
 with Lie algebra 
 ${\mathfrak{g}}_{\mathbb{C}}={\mathfrak{g}}\otimes_{\mathbb{R}}{\mathbb{C}}$, 
 and $H_{\mathbb{C}}$ a subgroup of $G_{\mathbb{C}}$
 with complexified Lie algebra 
 ${\mathfrak{h}}_{\mathbb{C}}={\mathfrak{h}}\otimes_{\mathbb{R}}{\mathbb{C}}$.  Let $B_G$ and $B_H$ be Borel subgroups of $G_{\mathbb{C}}$ 
 and $H_{\mathbb{C}}$, 
 respectively.  

\begin{df}
\label{def:BB}
{\rm{
We say the pair
 $(G,H)$
 (or the pair $({\mathfrak{g}}, {\mathfrak{h}})$)
 satisfies {\rm{(BB)}}
 if one of the following equivalent conditions is satisfied:
\begin{enumerate}
\item[{\rm{(BB1)}}]
$(G_{\mathbb{C}} \times H_{\mathbb{C}})/
\operatorname{diag}H_{\mathbb{C}}$
 is spherical 
 as a $(G_{\mathbb{C}} \times H_{\mathbb{C}})$-space.   
\item[{\rm{(BB2)}}]
$G_{\mathbb{C}}/B_H$ is spherical 
 as a $G_{\mathbb{C}}$-space.  
\item[{\rm{(BB3)}}]
$G_{\mathbb{C}}$ has an open orbit in $G_{\mathbb{C}}/B_G \times G_{\mathbb{C}}/B_H$
 via the diagonal action.  
\item[{\rm{(BB4)}}]
There are finitely many $G_{\mathbb{C}}$-orbits
 in $G_{\mathbb{C}}/B_G \times G_{\mathbb{C}}/B_H$
 via the diagonal action.  
\end{enumerate}
}}
\end{df}
It follows from \cite[Lemmas 4.2 and 5.3]{xtoshitoshima}
 that we have the implication
\[
  \text{(BB)} \Rightarrow \text{(PP)}.  
\]

Among the pairs $({\mathfrak {g}}, {\mathfrak {h}})$
 in Theorem \ref{thm:1.1}
 satisfying (PP), 
 we list the pairs $({\mathfrak {g}}, {\mathfrak {h}})$ 
 satisfying (BB) as follows:

\begin{prop}
\label{prop:B}
Suppose $({\mathfrak {g}}, {\mathfrak {h}})$ is a 
 reductive symmetric pair.  
Then the following conditions
 are equivalent:
\begin{enumerate}
\item[{\rm{(i)}}]
$({\mathfrak {g}}, {\mathfrak {h}})$
 satisfies {\rm{(BB)}}.

\item[{\rm{(ii)}}]
The pair of the Lie algebras 
 $({\mathfrak{g}},{\mathfrak{h}})$
 is isomorphic 
 {\rm{(}}up to outer automorphisms{\rm{)}}
 to the direct sum of pairs {\rm{(A)}}, {\rm{(B)}}
 and {\rm{(F1)}} -- {\rm{(F5)}}.  
\end{enumerate}
\end{prop}

\begin{remark}
\label{rem:1.6}
{\rm{
The classification in Theorem \ref{thm:1.1}
 (or in the \lq\lq{irreducible case}\rq\rq, 
 see Theorem \ref{thm:1.2})
 was known earlier in the following special cases:
\begin{enumerate}
\item[{\rm{1)}}]
$({\mathfrak {g}}, {\mathfrak {h}})$
complex pairs:
\quad
{\rm{(BB)}} $\Leftrightarrow$ 
{\rm{(F1)}} or {\rm{(F2)}}
\quad
(M. Kr{\"a}mer \cite{Kr}).  
\item[{\rm{2)}}]
$\operatorname{rank}_{\mathbb{R}}G=1$:
(PP) $\Leftrightarrow$  {\rm{(E1)}} -- {\rm{(E4)}}
 \quad
(B. Kimelfeld \cite{Kimelfeld}).  
\item[{\rm{2)}}]
$({\mathfrak {g}}, {\mathfrak {h}})
=({\mathfrak {g}}'+{\mathfrak {g}}', \operatorname{diag}
{\mathfrak {g}}')$ 
(\lq\lq{group case}\rq\rq):
\quad
{\rm{(PP)}} $\Leftrightarrow$ 
{\rm{(G1)}} or {\rm{(G2)}}
\quad
(\cite{xtoshi95}).  
\end{enumerate}
}}
\end{remark}

Concerning Remark \ref{rem:1.6} (2), 
 neither the concept (PP) nor a minimal parabolic subgroup
 of $H$ appeared 
 in \cite{Kimelfeld}, 
 but one might read (E1)--(E4) from his work.  
The case (1)
 (\lq\lq{strong Gelfand pairs}\rq\rq)
 was studied
 in connection with finite-dimensional representations
 of compact Lie groups,
 and the \lq\lq{group case}\rq\rq (3)
 with the tensor product
 of two (infinite-dimensional) representations, 
 see Corollary \ref{cor:1.2-copy}
 for more details.  

\vskip 1pc
The significance
 of these geometric conditions
 (PP) and (BB) is
 their applications to branching problems
 of infinite-dimensional representations
 of real reductive groups $G$ to subgroups $H$:

\begin{fact}
[{\cite[Theorems C and D]{xtoshitoshima}}]
\label{fact:1.4}
Suppose $G$ is a real reductive Lie group,
 and $H$ a reductive subgroup 
 defined algebraically over ${\mathbb{R}}$.  
\begin{enumerate}
\item[{\rm{1)}}]
{\rm{(}}Finite-multiplicity for branching{\rm{)}}
The pair $(G,H)$ satisfies {\rm{(PP)}}
 if and only if 
\[
  \dim \operatorname{Hom}_H(\pi|_H, \tau)<\infty
\]
for any admissible smooth representation
 $\pi$ of $G$
 and for any admissible smooth representation $\tau$ of $H$.  
\item[{\rm{2)}}]
{\rm{(}}Bounded-multiplicity for branching{\rm{)}}
The pair $(G,H)$ satisfies {\rm{(BB)}}
 if and only if 
 there exists a constant $C< \infty$
such that 
\[
  \dim \operatorname{Hom}_H(\pi|_H, \tau) \le C
\]
for any irreducible smooth representation
 $\pi$ of $G$
 and for any irreducible smooth representation $\tau$ of $H$.  
\end{enumerate}
\end{fact}
In Section \ref{sec:fm}
 we review briefly some basic notion
on admissible smooth representations of real reductive groups
 and discuss applications
 of our classification results
 to branching problems
 in details.  

\vskip 1pc
\par\noindent
{\bf{Organization of the paper.}}
\enspace
We give an outline
 of the proof of Theorem \ref{thm:1.1}
 in Section \ref{sec:strategy}
 by dividing it into five steps.  
Sections \ref{sec:method} to \ref{sec:Ke}
 are devoted to the proof of Theorem \ref{thm:1.1} of the paper.

In Section \ref{sec:fm} we explain 
 our initial motivation 
 for studying the real spherical property (PP) from 
 the viewpoint of the (infinite-dimensional) representation theory
 of real reductive groups,
 and give an application
 of our geometric results 
 to branching problems
 of smooth admissible representations.

\vskip 1pc
{\bf{Notation:}}\enspace
${\mathbb{R}}_+ :=\{t \in {\mathbb{R}}:t >0\}$, 
 and ${\mathbb{R}}_{\ge 0} :=\{t \in {\mathbb{R}}:t \ge0\}$.  
\vskip 1pc
{\bf{Acknowledgement:}}\enspace
The first author was partially supported
 by Grant-in-Aid for Scientific Research (A)(25247006) JSPS.

\vskip 1pc
\section{Strategy of the proof}
\label{sec:strategy}

We give an outline
 of the proof of Theorem \ref{thm:1.1}
 by dividing it into five steps.

\par\noindent
{{\bf{Step 1.}}}\enspace
Reduction to irreducible symmetric pairs.  
\par
A reductive symmetric pair
 $({\mathfrak {g}}, {\mathfrak {h}})$
 is said to be {\it{irreducible}}
 if ${\mathfrak {g}} \not \simeq {\mathbb{R}}$, 
 ${\mathfrak{h}}$ 
 and if $({\mathfrak {g}}, {\mathfrak {h}})$
 is not isomorphic to the direct sum
 of two reductive symmetric pairs
 $({\mathfrak {g}}_1, {\mathfrak {h}}_1)$
 and $({\mathfrak {g}}_2, {\mathfrak {h}}_2)$.  
The proof of Theorem \ref{thm:1.1} reduces
 to the case where $({\mathfrak {g}}, {\mathfrak {h}})$
 is an irreducible symmetric pair.  
This consists of two families
 up to outer automorphisms:
\begin{enumerate}
\item[1)]
(group case)
$({\mathfrak{g}}'+{\mathfrak{g}}',\operatorname{diag}{\mathfrak{g}}')$
with 
${\mathfrak{g}}'$ simple.  
\item[2)]
$({\mathfrak{g}},{\mathfrak{h}})$
with 
${\mathfrak{g}}$ simple.  
\end{enumerate}
Therefore the task of this article
 is to carry out the following classification:
\begin{thm}
\label{thm:1.2}
For an irreducible symmetric pair
 $({\mathfrak {g}}, {\mathfrak {h}})$, 
 the following two conditions are equivalent:
\begin{enumerate}
\item[{\rm{(i)}}]
$(G \times H)/\operatorname{diag} H$
 is real spherical.  
\item[{\rm{(ii)}}]
$({\mathfrak {g}}, {\mathfrak {h}})$ is isomorphic 
 to one of {\rm{(C)--(H)}}
 up to outer automorphisms.  
\end{enumerate}
\end{thm}

The main case is when ${\mathfrak {g}}$
 is simple.  
The {\lq\lq{group case}\rq\rq}\ (G) is relatively easy
 and the classification of those satisfying (PP)
 was already given in \cite{xtoshi95},
 but we supply a proof here
 for the sake of completeness.  
\vskip 1pc
\par\noindent
{\bf{Step 2.}}\enspace
Condition (QP).  

\par
Suppose $\sigma$ is an involutive automorphism of $G$.  
In general there is no $\sigma$-stable minimal parabolic subgroup of $G$.  
We introduce a condition (QP)
 which is slightly weaker
 than (PP)
 by using a $\sigma$-stable parabolic subgroup 
 (Subsection \ref{subsec:Q}).  
The difference between (QP) and (PP) 
 is described as in Theorem \ref{thm:QpPp} below.    

\vskip 1pc
\par\noindent
{\bf{Step 3.}}\enspace
Linearization of (PP) and (QP).  
\par
We find a necessary and sufficient condition 
 for a pair $(G,H)$
 to satisfy the conditions (PP) 
 (and also (QP)), 
 by means of the open-orbit property
 of a certain linear action
 (Theorems \ref{thm:pp} and \ref{thm:qp}).  
The case (QP) is easier
 because the parabolic subgroup $Q$ is $\sigma$-stable, 
 whereas the criterion of (PP)
 is more involved
 since the parabolic subgroup $P_G$
 (or its conjugate) 
 is not necessarily $\sigma$-stable.  

\vskip 1pc
\par\noindent
{\bf{Step 4.}}\enspace
The proof for (ii) $\Rightarrow$ (i)
 in Theorem \ref{thm:1.1}.  

The proof is carried out 
 in Sections \ref{sec:cpx}, \ref{sec:classical} and \ref{sec:rank1}.  
We shall verify the existence
 of an open orbit
 of the adjoint action of $(M_H \cap M_G)A_H$
 in ${\mathfrak {n}}^{-\sigma}$,
 by using the criterion of (PP)
 in Step 3.

Section \ref{sec:classical} deals
 with specific symmetric pairs
 in a case-by-case fashion,
 for instance, 
 $({\mathfrak {g}}, {\mathfrak {h}})=
  ({\mathfrak {o}}(i+j, k+l), {\mathfrak {o}}(i,k)+{\mathfrak {o}}(j,l))$.  
We use the invariant theory
 of quivers
 (Subsection \ref{subsec:upq}).  
In the section we classify 
 not only the irreducible symmetric pairs
 $({\mathfrak {g}}, {\mathfrak {h}})$
 satisfying (PP)
 but also those satisfying (QP).  
See Theorem \ref{thm:QpPp} below.  
Here is the precise place
 where the proof for (ii) $\Rightarrow$ (i)
 in Theorem \ref{thm:1.1}
 is given.  
We give two proofs
 for some families of symmetric pairs $(G,H)$.  

\begin{alignat*}{2}
&\text{Proposition \ref{prop:cpx}}
\qquad
&&
\text{(F1)(F2)(F3)(F4)(F5)}
\\
&\text{Proposition \ref{prop:upq}}
\qquad
&&
\text{(E1)(E2)(E3)(F4)(F5)(H4)}
\\
&\text{Proposition \ref{prop:glgl}}
\qquad
&&
\text{(F1)(F3)(H2)}
\\
&\text{Proposition \ref{prop:somn}}
\qquad
&&
\text{(F2)}
\\
&\text{Proposition \ref{prop:sostar}}
\qquad
&&
\text{(H3)}
\\
&\text{Proposition \ref{prop:e6}}
\quad
&&
\text{(H5)}
\\
&\text{Proposition \ref{prop:rankH}}
\qquad
&&
\text{(E1)(E2)(E3)(E4)}
\end{alignat*}

\par\noindent
{\bf{Step 5.}}\enspace
The proof for (i) $\Rightarrow$ (ii)
 in Theorem \ref{thm:1.1}.  

The proof is carried out 
together with the classification
of a larger set
 of the irreducible symmetric pairs
 satisfying (QP).  
We divide irreducible symmetric pairs
 $({\mathfrak {g}}, {\mathfrak {h}})$
 into the following three cases.  
Some of the results
 for concrete examples 
 in Section \ref{sec:classical}
 are used in Sections \ref{sec:nonKe} and \ref{sec:Ke}.  
\begin{alignat*}{2}
\text{Case 5a. \enspace (Section \ref{sec:rank1})}
\qquad
&\operatorname{rank}_{\mathbb{R}}H=1.  
&&
\\
\text{Case 5b. \enspace (Section \ref{sec:nonKe})}
\qquad
&\operatorname{rank}_{\mathbb{R}}H \ge 2, 
\qquad
&&\text{$({\mathfrak {g}}, {\mathfrak {g}}^{\sigma \theta})$
does not belong to $K_{\varepsilon}$-family.}
\\
\text{Case 5c. \enspace (Section \ref{sec:Ke})}
\qquad
&\operatorname{rank}_{\mathbb{R}}H \ge 2, 
\qquad
&&\text{$({\mathfrak {g}}, {\mathfrak {g}}^{\sigma \theta})$
belongs to $K_{\varepsilon}$-family.}
\end{alignat*}
\vskip 1pc
As a byproduct
 of the proof of Theorem \ref{thm:1.1}, 
 we obtain a complete list
 of the irreducible symmetric pairs
 $({\mathfrak {g}}, {\mathfrak {h}})$
 satisfying (QP):
\begin{thm}
\label{thm:QpPp}
Irreducible symmetric pairs
 satisfying {\rm{(QP)}} but not satisfying {\rm{(PP)}}
 are listed as follows:
\begin{alignat*}{3}
\operatorname{I_{\mathbb{R}}}&:
\enspace
&&({\mathfrak{o}} (p+1,q), {\mathfrak{o}} (p)+{\mathfrak{o}} (1,q))
&&(p,q \ge 2), 
\\
\operatorname{I_{\mathbb{C}}}&:
\enspace
&&({\mathfrak{su}} (p+1,q), {\mathfrak{s}}({\mathfrak {u}}(p)+{\mathfrak{u}} (1,q)))
\qquad
&&
(p,q \ge 2), 
\hphantom{MMMMMMMMMMMMM}
\\
\operatorname{I_{\mathbb{H}}}&:\enspace
&&({\mathfrak{sp}} (p+1,q), {\mathfrak{sp}} (p)+{\mathfrak{sp}} (1,q))
&&(p,q \ge 2)
\\
{\operatorname{II}}&:
\enspace
&&({\mathfrak{o}} (n+1, {\mathbb{C}}), {\mathfrak{o}} (n,1))
&&(n \ge 4)
\\
{\operatorname{III}}&:
\enspace
&&({\mathfrak{o}}^{\ast} (2n+2), {\mathfrak{u}} (n,1))
&&(n \ge 4).  
\end{alignat*}
\end{thm}
\begin{proof}[Outline of proof]
We shall see in  Proposition \ref{prop:rankH}
 that $({\mathfrak {g}}, {\mathfrak {h}})$
 satisfies (QP)
 if $({\mathfrak {g}}, {\mathfrak {h}})$
 is one of ${\rm{I}}_{\mathbb{R}}$,
 ${\rm{I}}_{\mathbb{C}}$,
 ${\rm{I}}_{\mathbb{H}}$,
 ${\rm{II}}$ or ${\rm{III}}$.  
Parts of this assertion
 also follow from Proposition \ref{prop:upq}
 when $({\mathfrak {g}}, {\mathfrak {h}})$
 is ${\rm{I}}_{\mathbb{R}}$, 
 ${\rm{I}}_{\mathbb{C}}$,
 or ${\rm{I}}_{\mathbb{H}}$, 
 and from Proposition \ref{prop:OUpq}
 when $({\mathfrak {g}}, {\mathfrak {h}})$
 is ${\rm{III}}$.  
The exhaustion
 of this list is a crucial part of Step 5.  
\end{proof}
By the classification in Theorem \ref{thm:QpPp}, 
we obtain
\begin{cor}
\label{cor:QP}
For irreducible symmetric pairs
 $({\mathfrak{g}}, {\mathfrak {h}})$
 with $\operatorname{rank}_{\mathbb{R}} {\mathfrak{h}} \ge 2$, 
 {\rm{(PP)}} $\Leftrightarrow$ {\rm{(QP)}}.  
\end{cor}
\section{Linearization of the open-orbit conditions 
 {\rm{(PP)}} and {\rm{(QP)}}}
\label{sec:method}

The goal of this section
is to give a criterion for {\rm{(PP)}}
 by linearization.  
The main result is Theorem \ref{thm:pp}.  
The proof for the implication (ii) $\Rightarrow$ (i)
 in Theorem \ref{thm:1.2} is carried
 out by this criterion
 in later sections.  
In order to optimise the proof 
 for the exhaustion of the list (C)--(H), 
 we introduce another geometric condition (QP), 
 which is slightly weaker than (PP).  
Then the condition {\rm{(QP)}} becomes a stepping-stone
 in the proof of the implication (i) $\Rightarrow$ (ii)
 by removing most of the symmetric pairs
 $({\mathfrak {g}}, {\mathfrak {h}})$
 that do not satisfy {\rm{(PP)}}.   
The condition {\rm{(QP)}} is also linearized 
 in Theorem \ref{thm:qp}.  
We shall further analyze the condition {\rm{(QP)}}
 in Propositions \ref{prop:cdual} and \ref{prop:QPrank}.  

\subsection{Parabolic subgroup $Q$
 associated to $(G,H)$}
\label{subsec:Q}

Let $G$ be a real reductive linear Lie group.  
Suppose that $\sigma$ is an involutive automorphism
 of $G$.  
The set $G^{\sigma}:=\{g \in G: \sigma g =g\}$
 of fixed points
 by $\sigma$
 is a closed subgroup of $G$.  
We say $(G,H)$ is a {\it{reductive symmetric pair}}
 if $H$ is an open subgroup of $G^{\sigma}$
 for some $\sigma$.

We take a Cartan involution $\theta$ of $G$
 commuting with $\sigma$, 
and set $K:=G^{\theta}$, 
 a maximal compact subgroup of $G$. 
The Lie algebras will be denoted by 
lower German letters
 such as ${\mathfrak {g}}$,
 ${\mathfrak {h}}$, 
 ${\mathfrak {k}}$, 
 $\cdots$, 
 and we shall use the same letters
 $\sigma$ and $\theta$
 for the induced automorphisms of the Lie algebra ${\mathfrak {g}}$.  
If $\tau$ is an involutive endomorphism
 of a real vector space $V$, 
 then $\tau$ is diagonalizable with eigenvalues $\pm 1$.  
We write the eigenspace
 decomposition as
\[
V=V^{\tau} +V^{-\tau}
\]
 where 
 $V^{\pm \tau}:=\{X \in V:\tau X=\pm X\}$.  
With the above notation, 
 ${\mathfrak {h}}={\mathfrak {g}}^{\sigma}$, 
 ${\mathfrak {k}}={\mathfrak {g}}^{\theta}$, 
 and ${\mathfrak {g}}={\mathfrak {g}}^{\theta} + {\mathfrak {g}}^{-\theta}$
 is a Cartan decomposition of 
the Lie algebra ${\mathfrak {g}}$.

We fix a maximal abelian subspace 
${\mathfrak {a}}_H$
 in ${\mathfrak {h}}^{-\theta}$, 
 and extend it to a maximal abelian subspace 
 ${\mathfrak {a}}_G$ 
 in ${\mathfrak {g}}^{-\theta}$.  
The split rank of $H$
 will be denoted by 
\[
\operatorname{rank}_{\mathbb{R}}H
:=
\dim_{\mathbb{R}}{\mathfrak {a}}_H.
\]
For $\alpha \in {\mathfrak {a}}_G^{\ast}$, 
 we write 
\[
  {\mathfrak {g}}({\mathfrak {a}}_G; \alpha)
:=\{X \in {\mathfrak {g}}
   :
   [H,X]=\alpha(H)X
   \quad
   \text{ for }\,\, H \in {\mathfrak {a}}_G
  \}, 
\]
 and denote by 
 $\Sigma({\mathfrak {g}}, {\mathfrak {a}}_G)$
 the set of nonzero $\alpha$
 such that ${\mathfrak {g}}({\mathfrak {a}}_G; \alpha)\ne \{0\}$.  
Similar notation 
is applied to ${\mathfrak {a}}_H$.  
Then the set of nonzero weights
 $\Sigma({\mathfrak {g}}, {\mathfrak {a}}_H)$
 satisfies
 the axiom of root systems
 (\cite[Theorem 2.1]{OS})
 as well as $\Sigma({\mathfrak {g}}, {\mathfrak {a}}_G)$.  
We choose {\it{compatible}} positive systems
 $\Sigma^+({\mathfrak {g}}, {\mathfrak {a}}_G)$
 and $\Sigma^+({\mathfrak {g}}, {\mathfrak {a}}_H)$
 in the sense
 that 
\[
  \alpha|_{\mathfrak {a}_H}
 \in \Sigma^+({\mathfrak {g}}, {\mathfrak {a}}_H) \cup \{0\}
\text{ for any }
 \alpha \in \Sigma^+({\mathfrak {g}}, {\mathfrak {a}}_G).  
\]
We write ${\mathfrak {g}}({\mathfrak {a}}_G; \alpha)$
 and ${\mathfrak {g}}({\mathfrak {a}}_H; \lambda)$
 for the root space
 of $\alpha \in \Sigma({\mathfrak {g}}, {\mathfrak {a}}_G)$
 and $\lambda \in \Sigma({\mathfrak {g}}, {\mathfrak {a}}_H)$, 
 respectively.  
We set
\begin{align*}
{\mathfrak {n}}:=&
\bigoplus 
_{\alpha|_{{\mathfrak {a}_H}} \in \Sigma^+({\mathfrak {g}}, {\mathfrak {a}}_H)}
{\mathfrak {g}}({\mathfrak {a}}_G; \alpha)
=
\bigoplus 
_{\lambda \in \Sigma^+({\mathfrak {g}}, {\mathfrak {a}}_H)}
{\mathfrak {g}}({\mathfrak {a}}_H; \lambda), 
\\
{\mathfrak {n}}_G:=&
\bigoplus 
_{\alpha \in \Sigma^+({\mathfrak {g}}, {\mathfrak {a}}_G)}
{\mathfrak {g}}({\mathfrak {a}}_G; \alpha).  
\end{align*}
Clearly 
 ${\mathfrak {n}} \subset {\mathfrak {n}}_G$.  
We remark
 that ${\mathfrak {n}}_G$ is not necessarily $\sigma$-stable,
 but ${\mathfrak {n}}$ is $\sigma$-stable.  
So we have a direct sum decomposition:
\[
   {\mathfrak {n}}={\mathfrak {n}}^{\sigma}+{\mathfrak {n}}^{-\sigma}.  
\]
We write $\Delta({\mathfrak {n}}^{\pm\sigma})$
 for the set of ${\mathfrak {a}}_H$-weights
 in ${\mathfrak {n}}^{\pm\sigma}$.  
Then we have
\[
  \Sigma^+({\mathfrak {g}}, {\mathfrak {a}}_H)
  =
  \Delta({\mathfrak {n}}^{\sigma}) \cup \Delta({\mathfrak {n}}^{-\sigma}), 
\]
 which is not disjoint in general.  
Let $P_G$ be the minimal parabolic subgroup
 of $G$
 that normalizes ${\mathfrak {n}}_G$, 
 $\overline{P_G}$ the opposite parabolic, 
 and $Q$ and $\overline Q$ the parabolic subgroups
 of $G$ 
 corresponding to $\Sigma^+({\mathfrak {g}}, {\mathfrak {a}}_H)$
 and $-\Sigma^+({\mathfrak {g}}, {\mathfrak {a}}_H)$, 
 respectively.  
Then ${\mathfrak {p}}_H:={\mathfrak {q}} \cap {\mathfrak {h}}$
 is a minimal parabolic subalgebra of ${\mathfrak {h}}$.  
We set 
\begin{align*}
M_G:=& Z_{K}({\mathfrak {a}}_G), 
\\
M_H:=& Z_{H \cap K}({\mathfrak {a}}_H), 
\\
A_H:=&\exp({\mathfrak {a}}_H), 
\\
L:=& Z_G({\mathfrak {a}}_H), 
L_H:=Z_H({\mathfrak {a}}_H)=M_H A_H.  
\intertext{Then we have}
\\
Q=&LN=L \exp ({\mathfrak {n}}), 
\quad
P_H=L_H N ^{\sigma}=M_HA_HN^{\sigma}.  
\end{align*}
We note 
 $P_G \subset Q \supset P_H$
 and $Q \cap \overline{P_G} = L \cap \overline{P_G}$.  

In Introduction,
 we considered the following two properties:
\begin{align*}
&\text{({\rm{PP}})\quad
$P_H$ has an open orbit on the real flag variety $G/P_G$, 
}
\\
&
\text{({\rm{BB}})\quad
$B_H$ has an open orbit 
 on the complex flag variety $G_{\mathbb{C}}/B_G$.  
}
\intertext{
We note that the conditions (PP) and (BB)
 are independent
 of coverings or connectedness
 of the groups,
 and depend only on the pair
 of the Lie algebras $({\mathfrak {g}}, {\mathfrak {h}})$.  
\vskip 0.3pc
In addition to the properties
 {\rm{(PP)}} and {\rm{(BB)}},
 we consider
}
&
\text{
{\rm{(QP)}}
\quad
$P_H$ has an open orbit
 on the real generalized flag variety $G/\overline Q$.  
}
\end{align*}
Among the three properties,
 we have:
\begin{lemma}
\label{lem:BPQ}
Let $({\mathfrak {g}}, {\mathfrak {h}})$ be a symmetric pair.  
Then we have
\begin{enumerate}
\item[{\rm{1)}}]
{\rm{(BB)}} $\Rightarrow$ {\rm{(PP)}} $\Rightarrow$ {\rm{(QP)}}.  
\item[{\rm{2)}}]
If $\operatorname{rank}_{\mathbb{R}}H
=\operatorname{rank}_{\mathbb{R}}G$, 
then {\rm{(PP)}} $\Leftrightarrow$ {\rm{(QP)}}.   
\end{enumerate}
\end{lemma}
\begin{proof}
1)\enspace
The implication {\rm{(BB)}} $\Rightarrow$ {\rm{(PP)}} 
 follows from \cite[Lemmas 4.2 and 5.3]{xtoshitoshima}.  
The implication {\rm{(PP)}} $\Rightarrow$ {\rm{(QP)}} is obvious 
 because $P_G$ is conjugate to $\overline P_G$
 and $\overline P_G \subset \overline Q$.  
\par\noindent
2)\enspace
If ${\mathfrak {a}}_H={\mathfrak {a}}_G$, 
then $P_G$ coincides with $Q$.  
Thus {\rm{(PP)}} is equivalent to {\rm{(QP)}}.  
\end{proof}

\begin{remark}
{\rm{
1)\enspace
We defined (PP) and (BB) 
 without assuming
 that $({\mathfrak {g}}, {\mathfrak {h}})$
 is a symmetric pair,
 however,
 we can define (QP)
 only for symmetric pairs
 $({\mathfrak {g}}, {\mathfrak {h}})$.  
\par\noindent
2)\enspace
The equivalence (PP) $\Leftrightarrow$ (QP)
 holds also 
 for any irreducible symmetric pair
 $({\mathfrak {g}}, {\mathfrak {h}})$
 with ${\operatorname{rank}}_{\mathbb{R}} {\mathfrak {h}} \ge 2$
 (see Corollary \ref{cor:QP}).  
}}
\end{remark}

\subsection{Criterion for (PP) and (QP)}
\label{subsec:PPQP}
We are ready to state a necessary and sufficient condition
 for the property {\rm{(PP)}}, 
 and that for {\rm{(QP)}}
 in terms of the adjoint action
 of $Z_H({\mathfrak {a}}_H)=M_H A_H$
 on ${\mathfrak {n}}^{-\sigma}$.  

\begin{thm}
\label{thm:pp}
The following two conditions are equivalent:
\begin{enumerate}
\item[{\rm{(i)}}]
$({\mathfrak {g}}, {\mathfrak {h}})$
 satisfies {\rm{(PP)}}.  
\item[{\rm{(ii)}}]
$(M_H \cap M_G)A_H$ has an open orbit on ${\mathfrak {n}}^{-\sigma}$
 via the adjoint action.  
\end{enumerate}
\end{thm}

\begin{thm}
\label{thm:qp}
Let $(G,H)$ be a reductive symmetric pair.  
Then the following two conditions are equivalent:
\begin{enumerate}
\item[{\rm{(i)}}]
$({\mathfrak {g}}, {\mathfrak {h}})$
 satisfies {\rm{(QP)}}.  
\item[{\rm{(ii)}}]
$Z_H({\mathfrak {a}}_H)=M_H A_H$ has an open orbit
 on ${\mathfrak {n}}^{-\sigma}$.  
\end{enumerate}
\end{thm}

For the proof of Theorems \ref{thm:pp} and \ref{thm:qp}, 
 we need a basic structural result
 on the centralizer of ${\mathfrak {a}}_H$
 and ${\mathfrak {a}}_G$, 
 respectively.  
\begin{lemma}
\label{lem:LIwasawa}
{\rm{1)}}\enspace
$
Z_{{\mathfrak{h}} \cap {\mathfrak{k}}}({\mathfrak{a}}_H)
\cap
Z_{{\mathfrak{k}}}({\mathfrak{a}}_G)
=
Z_{{\mathfrak{h}}\cap{\mathfrak{k}}}({\mathfrak{a}}_G).  
$
\begin{enumerate}
\item[{\rm{2)}}]\enspace
$
Z_{{\mathfrak{h}} \cap {\mathfrak{k}}}({\mathfrak{a}}_H)
+
Z_{{\mathfrak{k}}}({\mathfrak{a}}_G)
=
Z_{{\mathfrak{k}}}({\mathfrak{a}}_H).  
$
\item[{\rm{3)}}]\enspace
$
Z_{{\mathfrak{g}}}({\mathfrak{a}}_H)
=
Z_{{\mathfrak{h}}\cap {\mathfrak{k}}}({\mathfrak{a}}_H)
+
(Z_{{\mathfrak{g}}}({\mathfrak{a}}_H) \cap 
\overline{\mathfrak{p}}_G)
$.  
\end{enumerate}
\end{lemma}
\begin{proof}
1)\enspace  
Clear from ${\mathfrak {a}}_H \subset {\mathfrak {a}}_G$.  
\par\noindent
2)\enspace
If $\alpha|_{{\mathfrak {a}}_H}=0$
 then $\sigma \theta \alpha=\alpha$, 
 and therefore the involution $\sigma \theta$ stabilizes 
 ${\mathfrak{g}}({\mathfrak{a}}_G;\alpha)$
 with $\alpha|_{{\mathfrak{a}}_H}=0$.  
Thus we have a direct sum decomposition 
\[
   {\mathfrak{g}}({\mathfrak{a}}_{G};\alpha)
   =
   {\mathfrak{g}}^{\sigma\theta}({\mathfrak{a}}_G;\alpha)
   +
   {\mathfrak{g}}^{-\sigma\theta}({\mathfrak{a}}_G;\alpha).  
\]
We claim that
$
   {\mathfrak{g}}^{-\sigma\theta}({\mathfrak{a}}_G;\alpha)
   =\{0\}
$
for any $\alpha \in \Sigma({\mathfrak{g}}, {\mathfrak{a}}_G)$
 with $\alpha|_{{\mathfrak {a}}_H}=0$.  
In fact, 
 suppose that a nonzero element
 $X \in {\mathfrak{g}}({\mathfrak{a}}_G;\alpha)$
 satisfies 
 $\sigma \theta X=-X$.  
Then $X+ \sigma X \ne 0$
 because $\sigma X \in {\mathfrak{g}}({\mathfrak{a}}_G; \sigma \alpha)
 ={\mathfrak {g}}({\mathfrak {a}}_G;-\alpha)$
 and ${\mathfrak{g}}({\mathfrak{a}}_G;\alpha)
 \cap {\mathfrak{g}}({\mathfrak{a}}_G;-\alpha)=\{0\}$
 if $\alpha \ne 0$.  
On the other hand,
\[
  X + \sigma X = X - \theta X 
  \in {\mathfrak{h}}^{-\theta}.  
\]
Since $[{\mathfrak{a}}_H, X + \sigma X]=\{0\}$,
 it contradicts the maximality
 of ${\mathfrak{a}}_H$
 as an abelian subspace 
 in ${\mathfrak{h}}^{-\theta}$.  
Thus we have shown the claim.

Therefore we have the following direct sum decomposition
\begin{equation*}
Z_{\mathfrak{g}}({\mathfrak{a}}_H)
=\bigoplus_{\alpha|_{{\mathfrak{a}}_H}=0}
  {\mathfrak{g}}({\mathfrak{a}}_G;\alpha)
={\mathfrak{g}}({\mathfrak{a}}_G;0)
  \oplus
  \bigoplus_{\substack{\alpha|_{{\mathfrak{a}}_H}=0 \\ \alpha \ne0}}
  {\mathfrak{g}}^{\sigma\theta}({\mathfrak{a}}_G;\alpha).  
\end{equation*}
Taking the intersection 
 with ${\mathfrak {k}}$, 
we get the identity
\begin{align*}
Z_{\mathfrak{k}}({\mathfrak{a}}_H)
=&
  ({\mathfrak{g}}({\mathfrak{a}}_G;0) \cap {\mathfrak{k}})
\oplus
  \bigoplus
_{\substack{\alpha|_{{\mathfrak{a}}_H}=0,\\\alpha \ne 0}}
  ({\mathfrak{g}}^{\sigma\theta}({\mathfrak{a}}_G;\alpha) \cap {\mathfrak{k}})  
\\
=&Z_{\mathfrak{k}}({\mathfrak{a}}_G)
  +
  Z_{\mathfrak {h} \cap \mathfrak{k}}({\mathfrak{a}}_H).  
\end{align*}
\par\noindent
3)\enspace
By the Iwasawa decomposition
 of the reductive subalgebra $Z_{\mathfrak{g}}({\mathfrak{a}}_H)$,
 we have
\[
Z_{\mathfrak{g}}({\mathfrak{a}}_H)
=
Z_{\mathfrak{k}}({\mathfrak{a}}_H)
+
(Z_{\mathfrak{g}}({\mathfrak{a}}_H)
 \cap 
\overline {\mathfrak{p}}_G).  
\]
Combining this with the second statement,
 we have
\[
Z_{\mathfrak{g}}({\mathfrak{a}}_H)
=
Z_{\mathfrak {h} \cap \mathfrak{k}}({\mathfrak{a}}_H)
+
Z_{\mathfrak{k}}({\mathfrak{a}}_G)
+
(Z_{\mathfrak{g}}({\mathfrak{a}}_H)
 \cap 
\overline {\mathfrak{p}}_G)
=
Z_{\mathfrak {h} \cap \mathfrak{k}}({\mathfrak{a}}_H)
+
(Z_{\mathfrak{g}}({\mathfrak{a}}_H)
 \cap 
\overline {\mathfrak{p}}_G)
\]
 because 
 $Z_{\mathfrak{k}}({\mathfrak{a}}_G)
 \subset 
 Z_{\mathfrak{g}}({\mathfrak{a}}_H)
 \cap 
\overline {\mathfrak{p}}_G$.  
\end{proof}
The following lemma is known
 (\cite{Benoist}), 
 but we give a proof 
 for the sake of completeness.  
\begin{lemma}
\label{lem:Nsigma}
Suppose $N$ is a simply connected nilpotent Lie group
 with an involutive automorphism $\sigma$.  
\begin{enumerate}
\item[{\rm{1)}}]
The exponential map $\exp :{\mathfrak {n}} \to N$
induces bijections 
${\mathfrak {n}}^{\sigma} \overset \sim \to N^{\sigma}$
 and ${\mathfrak {n}}^{-\sigma} \overset \sim \to N^{-\sigma}$.  
\item[{\rm{2)}}]
The following map is also bijective:
\begin{equation}
\label{eqn:Nsigma}
{\mathfrak {n}}^{\sigma} + {\mathfrak {n}}^{-\sigma}
 \to N, 
\quad
(X, Y) \mapsto \exp X \exp Y.  
\end{equation}
\end{enumerate}
\end{lemma}

\begin{proof}
{\rm{1)}}\enspace
Since $N$ is simply connected and nilpotent, 
 the exponential map is bijective.  
We write $\log : N \to {\mathfrak {n}}$
 for its inverse.  
Then,
 for the first statement, 
 it is sufficient to prove
 the surjectivity 
 of the restriction ${\mathfrak {n}}^{\pm \sigma}
\to N^{\pm\sigma}$.  
Take an arbitrary $y \in N$
 such that $\sigma(y)=y^{\pm 1}$.  
Then $Y:=\log y$ satisfies
 $\exp (\sigma Y)=\exp(\pm Y)$, 
whence $\sigma Y = \pm Y$.  
Thus $\exp : {\mathfrak {n}}^{\sigma} \to N^{\sigma}$
 and ${\mathfrak {n}}^{-\sigma} \to N^{-\sigma}$
 are both surjective.  
\par\noindent
{\rm{2)}}\enspace
Clearly the map \eqref{eqn:Nsigma} is injective. 
To see \eqref{eqn:Nsigma} is surjective, 
 we take $z \in N$.  
Since $z^{-1} \sigma(z) \in N^{-\sigma}$,
we have $Y:=-\frac 1 2 \log(z^{-1} \sigma(z)) \in {\mathfrak {n}}
^{-\sigma}$.  
We set 
$x:=z \exp (-Y)$.  
Then 
$
   x \sigma(x)^{-1} = z \exp (-2 Y) \sigma(z)^{-1}=e.  
$
Thus $X:=\log (x) \in {\mathfrak {n}}^{\sigma}$
 and $z =\exp X \exp Y$.  
Hence we have shown 
 that the map \eqref{eqn:Nsigma} is surjective, 
 too.  
\end{proof}

\begin{lemma}
\label{lem:Qdouble}
We let $Z_H({\mathfrak {a}}_H)=M_H A_H$ act linearly
 on ${\mathfrak {n}}^{-\sigma}$.  
Then the natural inclusion 
 ${\mathfrak {n}}^{-\sigma} \overset {\exp}\to N^{-\sigma} 
\hookrightarrow Q$
 induces the following bijections:
\begin{align}
\label{eqn:Qdouble}
{\mathfrak {n}}^{-\sigma}/(M_H \cap M_G)A_H
&\overset \sim \to P_H\backslash Q / (L \cap \overline{P_G}),  
\\
\label{eqn:QLdouble}
{\mathfrak {n}}^{-\sigma}/M_H A_H
&\overset \sim \to P_H\backslash Q / L.  
\end{align}
\end{lemma}

\begin{proof}
It follows from Lemmas \ref{lem:LIwasawa} and \ref{lem:Nsigma}
 that 
\[Q=NL=NM_H(L \cap \overline{P_G})
=M_H N^{\sigma} \exp ({\mathfrak {n}}^{-\sigma})
(L \cap \overline{P_G}).  
\]
Thus the map \eqref{eqn:Qdouble} is surjective, 
 and so is \eqref{eqn:QLdouble}.  
\par\noindent
1)\enspace
Suppose two elements $X_1$, $X_2 \in {\mathfrak {n}}^{-\sigma}$
 have the same image
 in \eqref{eqn:Qdouble}.  
This means
 that there exist $l_H \in L_H =M_HA_H$, 
 $n_H \in N^{\sigma}$, 
 and $l \in L \cap \overline{P_G}$
 such that $x_i = \exp (X_i)$
 ($i=1,2$)
 satisfy $x_1=l_Hn_Hx_2 l$.  
Then we have
\[
  L \ni l_H^{-1}l^{-1}
     =(l_H^{-1} x_1^{-1} l_H)n_H x_2
     \in N^{-\sigma}N^{\sigma}N^{-\sigma}=N, 
\]
 and therefore $l=l_H^{-1}$, 
 $n_H =e$, 
 and $l_H^{-1}x_1 l_H=x_2$.  
Hence $\operatorname{Ad} (l_H)X_2=X_1$.  
Since $l_H=l^{-1}$ belongs
 to 
\[
M_H A_H \cap (L \cap \overline P_G)
=
(M_H \cap M_G)A_H, 
\]
the map \eqref{eqn:Qdouble} is injective.  
\par\noindent 
2)\enspace
The proof parallels to that for \eqref{eqn:Qdouble}.  
The only difference is that $l \in L$ 
 instead of the previous condition $l \in L \cap \overline P_G$, 
and thus $l_H=l^{-1}$ belongs to 
 $M_H A_H \cap L = M_H A_H$.  
Hence $X_1$ and $X_2$ give the same equivalence class 
under the action of $M_HA_H$.  
\end{proof}

We are ready 
 to complete the proof
 of Theorems \ref{thm:pp} and \ref{thm:qp}.  
\begin{proof}
[Proof of Theorem \ref{thm:pp}]
Since any minimal parabolic subgroup
 is conjugate to each other
 by inner automorphisms,
 {\rm{(PP)}} is equivalent to the existence
 of an open $P_H$-orbit in $G/\overline P_G$.  
By the Bruhat decomposition,
 the $Q$-orbit through the origin $o=e \overline P_G$
 in $G/\overline P_G$ is open dense
 because $P_G \subset Q$.  
This open orbit is given by  
 $Q/(Q \cap \overline P_G) = Q/(L \cap \overline P_G)$
 as a homogeneous space of $Q$.  
Since $Q$ contains $P_H$, 
 the condition {\rm{(PP)}} is equivalent 
 to the existence 
 of an open $P_H$-orbit in $Q/(L \cap \overline P_G)$.  
By Lemma \ref{lem:Qdouble}, 
this amounts to the existence
 of an open $(M_H \cap M_G)A_H$ orbit
 in ${\mathfrak {n}}^{-\sigma}$.  
\end{proof}

\begin{proof}
[Proof of Theorem \ref{thm:qp}]
The proof is similar to that for Theorem \ref{thm:pp}.  
In fact, 
 since the $Q$-orbit through the origin $o=e \overline Q$
 in $G/\overline Q$  
 is open dense 
 and given by 
 $Q/(Q \cap \overline Q) = Q/L$, 
the condition {\rm{(QP)}} is equivalent
 to the existence
 of an open $P_H$-orbit in $Q/L$, 
which in turn is equivalent to the existence
 of an open $M_H A_H$-orbit
 in ${\mathfrak {n}}^{-\sigma}$
 by Lemma \ref{lem:Qdouble}.  
\end{proof}

\subsection{$c$-dual of symmetric pairs and (QP)}
\label{subsec:cdual}
For a symmetric pair $({\mathfrak {g}}, {\mathfrak {h}})$
 defined
 by an involutive automorphism $\sigma$
 of ${\mathfrak {g}}$, 
we write 
\[
   {\mathfrak {g}}={\mathfrak {g}}^{\sigma} + {\mathfrak {g}}^{-\sigma}
\]
for the eigenspace decomposition
 of $\sigma$
 with eigenvalues $+1$ and $-1$
 as before.  
Then ${\mathfrak {h}}={\mathfrak {g}}^{\sigma}$.  
We set 
\[
{\mathfrak {g}}^c:={\mathfrak {g}}^{\sigma}+\sqrt{-1}{\mathfrak {g}}^{-\sigma}.  \]
Then the vector space ${\mathfrak {g}}^c$
 carries a natural Lie algebra structure,
 and the pair $({\mathfrak {g}}^c, {\mathfrak {h}})$
 forms a symmetric pair
 by the restriction of the complex linear extension 
 of $\sigma$
 to ${\mathfrak {g}}_{\mathbb{C}} = {\mathfrak {g}} \otimes _{\mathbb{R}} 
 {\mathbb{C}}$.  
The pair $({\mathfrak {g}}^c, {\mathfrak {h}})$
 is called the 
 {\it{$c$-dual}}
 of the symmetric pair 
 $({\mathfrak {g}}, {\mathfrak {h}})$.  
We note
 that ${\mathfrak {g}}$ is reductive 
 if and only if ${\mathfrak {g}}^c$ is reductive.  
\begin{exmp}
\label{ex:cdual}
{\rm{
1)\enspace
The $c$-dual of the \lq\lq{group case}\rq\rq\
 $({\mathfrak {g}}\oplus {\mathfrak {g}}, \operatorname{diag}{\mathfrak {g}})$
 is isomorphic to the pair 
 $({\mathfrak {g}}_{\mathbb{C}}, {\mathfrak {g}})$
 where the involution of ${\mathfrak {g}}_{\mathbb{C}}$
 is given by the complex conjugation 
with respect to the real form ${\mathfrak {g}}$.    
\par\noindent
2)\enspace
The complex symmetric pair
 $({\mathfrak {g}}_{\mathbb{C}}, {\mathfrak {h}}_{\mathbb{C}})$
 is self $c$-dual.  
}}
\end{exmp}

\begin{prop}
\label{prop:cdual}
A reductive symmetric pair
 $({\mathfrak {g}}, {\mathfrak {h}})$ satisfies 
{\rm{{\rm{(QP)}}}}
 if and only if the $c$-dual
 $({\mathfrak {g}}^c, {\mathfrak {h}})$ satisfies {\rm{(QP)}}.  
\end{prop}
\begin{proof}
By the criterion in Theorem \ref{thm:qp}, 
 the reductive symmetric pair
 $({\mathfrak {g}}, {\mathfrak {h}})$
 (respectively,
 the $c$-dual $({\mathfrak {g}}^c, {\mathfrak {h}})$)
 satisfies {\rm{(QP)}}
 if and only if 
 the group $Z_H({\mathfrak {a}}_H)$
 has an open orbit in ${\mathfrak {n}}^{-\sigma}$
 (respectively, 
 in $\sqrt{-1}{\mathfrak {n}}^{-\sigma}$).  
Since ${\mathfrak {n}}^{-\sigma}$
 and $\sqrt{-1}{\mathfrak {n}}^{-\sigma}$
 are isomorphic
 to each other
 as modules of the group $Z_H({\mathfrak {a}}_H)$, 
 we get the proposition.  
\end{proof}
\begin{remark}
\label{rem:cdual}
{\rm{
An analogous statement to Proposition \ref{prop:cdual}
 does not hold for {\rm{(PP)}}
 in general.  
}}
\end{remark}

\subsection{Further properties for (QP)}
\label{subsec:QPmore}
In order to screen the symmetric pairs
 that do not satisfy (QP), 
 it is convenient to find
 a necessary condition
 for (QP) 
 in terms of the restricted root system.

Here is the one 
 that we frequently use in later sections:
\begin{prop}
\label{prop:QPrank}
If $({\mathfrak {g}}, {\mathfrak {h}})$
 satisfies {\rm{(QP)}}, 
 then elements of $\Delta({\mathfrak {n}}^{-\sigma})$
 are linearly independent.  
In particular, 
 we have
\begin{equation}
\label{eqn:rn}
\operatorname{rank}_{\mathbb{R}}H 
\ge
\# \Delta({\mathfrak {n}}^{-\sigma}), 
\end{equation}
where $\# \Delta({\mathfrak {n}}^{-\sigma})$ denotes 
the cardinality 
 of the weights of ${\mathfrak {a}}_H$
 in ${\mathfrak {n}}^{-\sigma}$ 
 without counting the multiplicities.  
\end{prop}
The converse statement
 of Proposition \ref{prop:QPrank}
 is not true;
 however,
 we shall see 
 that the condition \eqref{eqn:rn} is a fairly good criterion
 for (QP).  
For example
 if $(G,H)=(SO^{\ast}(2p+2q),SO^{\ast}(2p) \times SO^{\ast}(2q))$
 then the condition \eqref{eqn:rn} is equivalent to (QP)
 except for $(p,q)=(2,2)$, 
 see Proposition \ref{prop:sostar}.

For $\lambda \in {\mathfrak {a}}_H^{\ast}
 ={\operatorname{Hom}}_{\mathbb{R}}({\mathfrak {a}}_H, {\mathbb{R}})$, 
let $\chi_{\lambda}$ be the one-dimensional real representation
 of the abelian group $A_H$
 given by
\[
\chi_\lambda(\exp Y)
=\exp \langle \lambda, Y \rangle
\quad
\text{ for }
Y \in {\mathfrak {a}}_H, 
\]
and write ${\mathbb{R}}_\lambda$
 $(\simeq {\mathbb{R}})$
 for the representation space
 of ${\chi}_{\lambda}$. 
\vskip 1pc
To prove the proposition,
 we need the following elementary lemma:
\begin{lemma}
\label{lem:li}
Let $F$ be a finite subset
 of ${\mathfrak {a}}_H^{\ast}$.  
If $A_H$ has an open orbit
 in the vector space $\bigoplus_{\lambda \in F}{\mathbb{R}}_\lambda$, 
 then $F$ consists of linearly independent elements.  
In particular, 
 $\# F \le \dim {\mathfrak {a}}_H$.  
\end{lemma}

We return to the proof of Proposition \ref{prop:QPrank}:
\begin{proof}
[Proof of Proposition \ref{prop:QPrank}]
Since ${\mathfrak {a}}_H$ is a maximal abelian subspace 
 in ${\mathfrak {h}}^{-\theta}$, 
 we have 
 $Z_H({\mathfrak {a}}_H)
 = M_H A_H$
 with $M_H=Z_{H \cap K}({\mathfrak {a}}_H)$
 compact.  
We equip ${\mathfrak {n}}^{-\sigma}$
with an $M_H$-inner product
 such that the decomposition 
\[{\mathfrak {n}}^{-\sigma}
\simeq
\bigoplus_{\lambda \in \Delta({\mathfrak {n}}^{-\sigma})}
{\mathfrak{g}}^{-\sigma}({\mathfrak{a}}_H;\lambda)
\]
 is orthogonal to each other.

Let $O_{\lambda}$ be the orthogonal group 
 of the subspace
 ${\mathfrak {g}}^{-\sigma}({\mathfrak {a}}_H;\lambda)$.  
Then the quotient space
 of ${\mathfrak {g}}^{-\sigma}({\mathfrak {a}}_H;\lambda)$
 by $O_{\lambda}$
 is given by the \lq\lq{half line}\rq\rq:
\[{\mathfrak {g}}^{-\sigma}({\mathfrak {a}}_H;\lambda)/O_{\lambda}
  \simeq ({\mathbb{R}}_{\lambda})_{\ge 0}.  
\]

Since the compact group $M_H$ preserves 
 the inner product on ${\mathfrak {n}}^{-\sigma}$, 
 we have a natural surjective map
 between the quotient spaces 
 of ${\mathfrak {n}}^{-\sigma}$:
\[
{\mathfrak {n}}^{-\sigma}/M_H
\to 
{\mathfrak {n}}^{-\sigma}
/
(\prod_{\lambda \in \Delta({\mathfrak {n}}^{-\sigma})} O_{\lambda})
\simeq 
\prod_{\lambda \in \Delta({\mathfrak {n}}^{-\sigma})}
({\mathfrak {g}}^{-\sigma}({\mathfrak {a}}_H;\lambda)/O_{\lambda})
\simeq 
\prod_{\lambda \in \Delta({\mathfrak {n}}^{-\sigma})}
(
{\mathbb{R}}_\lambda
)_{\ge 0}.  
\]
Then if $Z_{H \cap K}({\mathfrak {a}}_H)A_H$ has an open orbit 
 in ${\mathfrak {n}}^{-\sigma}$
 via the adjoint representation, 
 then $A_H$ has an open orbit in the quotient
 of ${\mathfrak {n}}^{-\sigma}$
 by $M_H$.  
Therefore $A_H$ has an open orbit
 in 
\[
\prod_{\lambda \in \Delta({\mathfrak {n}}^{-\sigma})}
 ({\mathbb{R}}_{\lambda})_{\ge 0}
 \subset 
\bigoplus_{\lambda \in \Delta({\mathfrak {n}}^{-\sigma})}
 {\mathbb{R}}_{\lambda},
\]
 too.  
Applying Lemma \ref{lem:li}, 
 we conclude
 that the elements of $\Delta({\mathfrak {n}}^{-\sigma})$
 are linearly independent
 and therefore, 
$\# \Delta({\mathfrak {n}}^{-\sigma}
)
\ge \dim {\mathfrak {a}}_H
=\operatorname{rank}_{\mathbb{R}}H$.  
Hence the proposition is proved.  
\end{proof}
Next we analyze the inequality \eqref{eqn:rn}
 in Proposition \ref{prop:QPrank}.  
For this, 
 we denote by $W_H$ the Weyl group 
 of the restricted root system 
 $\Sigma({\mathfrak {h}}, {\mathfrak {a}}_H)$. 
Then $W_H$ acts 
 on the finite set $\Delta({\mathfrak {n}}^{-\sigma})
\cup (-\Delta({\mathfrak {n}}^{-\sigma}))$, 
 and consequently, 
 we have an obvious inequality
\[
 2 \# \Delta({\mathfrak {n}}^{-\sigma})
 \ge \# (W_H \cdot \lambda)
\]
for any $\lambda \in \Delta({\mathfrak {n}}^{-\sigma})$.  
Hence
 the inequality \eqref{eqn:rn}
 implies
 that for any $\lambda \in \Delta({\mathfrak {n}}^{-\sigma})$, 
 we have 
\begin{equation}
\label{eqn:aW}
2 \dim {\mathfrak {a}}_H
\ge \# (W_H \cdot \lambda).  
\end{equation}
The inequality \eqref{eqn:aW} gives strong constraints
 on both the root system 
 $\Delta({\mathfrak{h}},{\mathfrak{a}}_H)$
 and $\Delta({\mathfrak{n}}^{-\sigma})$.  
Let us examine \eqref{eqn:aW}
 in an abstract setting 
(corresponding to the case
 where ${\mathfrak {h}}$ is simple)
as follows: 
\begin{lemma}
\label{lem:2.6}
Let $\Delta$ be an irreducible root system 
on a vector space $E$, 
and $W$ the Weyl group of $\Delta$.  
If there exists $\lambda \in E \setminus \{0\}$
 such that 
\begin{equation}
\label{eqn:EW}
2 \dim E \ge \# (W \cdot \lambda), 
\end{equation}
then $\Delta$ is a classical root system.  
For an (irreducible) classical root system $\Delta$, 
 we take a standard basis 
 and the set $\Pi$ of simple roots as follows:
\begin{align*}
\text{Case 1}:&
\Delta =A_n, 
\quad\hphantom{m}
\Pi =\{\alpha_i= e_i-e_{i+1}: 1 \le i \le n\}
\text{ in }
{\mathbb{R}}^n/{\mathbb{R}}(e_1 + \cdots + e_{n+1}), 
\\
\text{Case 2}:&
\Delta =
\begin{cases}
B_n, 
&\Pi=\{\alpha_i =e_i - e_{i+1}: 1 \le i \le n-1\} 
           \cup \{\alpha_n=e_n\}, 
\\
C_n, 
&\Pi=\{\alpha_i = e_i- e_{i+1}: 1 \le i \le n-1\}
 \cup \{\alpha_n=2e_n\}, 
\\
D_n,  
&\Pi=\{\alpha_i = e_i-e_{i+1}: 1 \le i \le n-1 \}
 \cup \{\alpha_n=e_{n-1}+e_n\}.  
\end{cases}
\end{align*}
Here we assume 
 $n \ge 1$ for $\Delta=A_n$, 
 $n \ge 2$ for $\Delta=B_n$,
 $n \ge 3$ for $\Delta=C_n$, 
 and $n \ge 4$ for $\Delta=D_n$.

Then $\lambda$ satisfying \eqref{eqn:EW}
 must be of the following form:
\begin{alignat*}{3}
& \lambda \in {\mathbb{R}}e_i /({\mathbb{R}}(e_1+\cdots +e_{n+1}))
\quad
&&\text{for some $i$ \,$(1 \le i \le n+1)$}
\quad
&&\text{in Case 1}, 
\\
& \lambda \in {\mathbb{R}}e_i
\quad
&&\text{for some $i$ \,\,$(1 \le i \le n)$}
&&\text{in Case 2}.  
\end{alignat*}

\end{lemma}
\begin{proof}
For a root system $\Delta$, 
 we consider the minimum cardinality
 of $W$-orbits 
 defined by 
\[
c(\Delta):=\inf_{\lambda \in E \setminus \{0\}}
\# (W \cdot \lambda).  
\]
Let us compute $c(\Delta)$.  
For this,
 we fix a positive system $\Delta^+$, 
 and write $\Pi=\{\alpha_1, \cdots, \alpha_n\}$
 for the set of simple roots, 
 and $\{\omega_1, \cdots, \omega_n\}$
 for the set of fundamental weights.  
In order to compute the cardinality
 of the orbit $W \cdot \lambda$, 
 we may assume $\lambda \in \overline{C_+}\setminus \{0\}$
 without loss of generality, 
 where $\overline{C_+}$ is the dominant chamber defined by 
\[
\overline{C_+}
:=\{\sum_{i=1}^{n} a_i \omega_i:
a_1, \cdots, a_n \ge 0\}.  
\]

We define a partial order on $\overline{C_+}$ by 
\[
 \lambda \succ \mu 
\quad
\text{if }
\lambda-\mu \in \overline{C_+}.  
\]
We denote by $W_{\lambda}$
 the isotropy subgroup
 of $W$ at $\lambda \in E$.  
Then $\# (W \cdot \lambda)=\# W/\# W_{\lambda}$.  
If $\lambda, \mu \in \overline{C_+}$
 satisfies $\lambda \succ \mu$, 
then there is an inclusion relation 
 $W_{\lambda} \subset W_{\mu}$,  
 and therefore
$\# W \cdot \lambda \ge \# W \cdot \mu$.  
Thus $\# W \cdot \lambda$ attains 
 its minimum 
 only if $\lambda$ lies in the most singular part of 
 the Weyl chamber, 
namely, 
 only if $\lambda \in {\mathbb{R}}_+ \omega_i$
 ($1 \le i \le n$).    
In this case,
 $W_{\lambda}$ coincides with the Weyl group 
 $W({\mathfrak {l}}_i)$
 of the Levi part ${\mathfrak {l}}_i$
 of the maximal parabolic subgroup 
 defined by the simple root $\alpha_i$.  
Thus we have
\[
  c(\Delta)=\frac{\# W}{\max_{1 \le i \le n} \# W({\mathfrak {l}}_i)}.  
\]
This formula yields the explicit value
 of $c(\Delta)$ 
 as in the table below,
 and also tells precisely when 
$\# W \cdot \lambda$ attains its minimum.  
\begin{table}[H]
\centering
\begin{tabular}{cccccccccc}
$\Delta$ \quad
&\, $A_n$ \,
&\, $B_n$ \,
&\, $C_n$ \,
&\, $D_n$ \,
&\, ${\mathfrak{e}}_6$ \,
&\, ${\mathfrak{e}}_7$ \,
&\, ${\mathfrak{e}}_8$ \,
&\, ${\mathfrak{f}}_4$ \,
&\, ${\mathfrak{g}}_2$ \,
\\
$c(\Delta)$ \quad
&\, $n+1$ \,\,
&\,\, $2n$ \,\,
&\,\, $2n$ \,\,
&\,\, $2n$ \,\,
&\,\, $27$ \,\,
&\,\, $56$ \,\,
&\,\, $240$ \,\,
&\,\, $24$ \,\,
&\,\, $6$ \,\,
\end{tabular}
\caption{$c(\Delta)$
 for simple root systems $\Delta$}
\label{tab:cDelta}
\end{table}

For $\Delta=A_n$, 
we label simple roots
 as indicated.  
Then by a simple computation,
 we see that 
 $\#W({\mathfrak {l}}_i)$ attains its maximum $n!$ 
 at $i=1$ and $i=n$, 
 and $\# W/\# W({\mathfrak {l}}_i) > 2n$
 for $2 \le i \le n-1$.  
Thus the inequality \eqref{eqn:EW} holds
 for $\lambda \in \overline{C_+} \setminus \{0\}$ 
 if and only if $\lambda \in {\mathbb{R}}_+ \omega_1$
 or ${\mathbb{R}}_+ \omega_n$, 
 namely,
 $\lambda \in {\mathbb{R}}_+ e_1$
 or $\lambda \in {\mathbb{R}}_- e_{n+1}
 \mod {\mathbb{R}}(e_1 + \cdots + e_{n+1})$.  

For $\Delta =B_n$
 ($n \ge 2$), 
$C_n$
 ($n \ge 3$), 
or 
$D_n$
 ($n \ge 4$), 
 $\#W({\mathfrak {l}}_i)$ attains its maximum only
 at $i=1$, 
 and the inequality \eqref{eqn:EW} is actually the equality
 when $\lambda \in {\mathbb{R}} \omega_1$.  

For exceptional root systems $\Delta$, 
it is immediate from Table \ref{tab:cDelta}
 that \eqref{eqn:EW} does not hold.  
Hence Lemma \ref{lem:2.6} is proved.  
\end{proof}

We end this section 
 with an easy-to-check necessary condition 
 for (QP)
 when $\operatorname{rank}_{\mathbb{R}} G=\operatorname{rank}_{\mathbb{R}}H$.  
Proposition \ref{prop:QPineq} below 
 will be used in Section \ref{sec:Ke}
 when we deal with exceptional Lie algebras.  
For a real reductive Lie group $G$
 with a Cartan involution $\theta$, 
 we take a maximal abelian subspace ${\mathfrak {a}}_G$
 in ${\mathfrak {g}}^{-\theta}$
 and fix a positive system $\Sigma^+({\mathfrak {g}}, {\mathfrak {a}}_G)$
 as before.  
We set 
\begin{align}
m(G):= \max_{\alpha \in \Sigma ({\mathfrak {g}}, {\mathfrak {a}}_G)}
       \dim_{\mathbb{R}}{\mathfrak {g}}({\mathfrak {a}}_G;\alpha), 
\label{eqn:mG}
\\
n(G):= \sum_{\alpha \in \Sigma^+ ({\mathfrak {g}}, {\mathfrak {a}}_G)}
       \dim_{\mathbb{R}}{\mathfrak {g}}({\mathfrak {a}}_G;\alpha).  
\label{eqn:nG}
\end{align}
We note that $n(G)$ is equal
 to the dimension of the real flag variety $G/P_G$.  

\begin{prop}
\label{prop:QPineq}
Assume
 $\operatorname{rank}_{\mathbb{R}}G=\operatorname{rank}_{\mathbb{R}}H$.  
If the symmetric pair $(G,H)$ satisfies {\rm{(QP)}}, 
then
\begin{equation}
\label{eqn:QPineq}
n(G) -n(H) \le m(G) \operatorname{rank}_{\mathbb{R}}H.  
\end{equation}
\end{prop}
\begin{proof}
Since ${\mathfrak{a}}_H={\mathfrak{a}}_G$
 by the real rank assumption,
 we have
\[m(G) \# \Delta({\mathfrak{n}}^{-\sigma})
 \ge \dim {\mathfrak{n}}^{-\sigma}
 =n(G)-n(H).
\]  
Hence the inequality \eqref{eqn:QPineq}
implies $\operatorname{rank}_{\mathbb{R}}
 H \ge \# \Delta ({\mathfrak{n}}^{-\sigma})$.  
Thus Proposition \ref{prop:QPineq}
 follows from Proposition \ref{prop:QPrank}.  
\end{proof}

\section{Strong Gelfand pairs 
 and their real forms}
\label{sec:cpx}
This section focuses on (BB),  
which is much stronger than (PP)
 for real reductive pair
 $({\mathfrak {g}}, {\mathfrak {h}})$
 in general 
unless both ${\mathfrak {g}}$ and ${\mathfrak {h}}$
 are quasi-split Lie algebras.  
We begin with the \lq\lq{complex case}\rq\rq.  
In this case the condition (BB) is also referred
 to as a {\it{strong Gelfand pair}}.  
\begin{prop}
\label{prop:cpx-1}
Suppose that 
 $({\mathfrak {g}}, {\mathfrak {h}})$
 is a symmetric pair
 such that ${\mathfrak {g}}$ is a complex simple Lie algebra
 and ${\mathfrak {h}}$ is a complex subalgebra.  
Then the following three conditions are equivalent:
\begin{enumerate}
\item[{\rm{(i)}}]
The pair $({\mathfrak {g}}, {\mathfrak {h}})$
satisfies {\rm{(PP)}}.  
\item[{\rm{(ii)}}]
The pair $({\mathfrak {g}}, {\mathfrak {h}})$
satisfies {\rm{(BB)}}.  
\item[{\rm{(iii)}}]
$({\mathfrak {g}}, {\mathfrak {h}})$ is isomorphic 
 to $({\mathfrak {sl}}(n+1, {\mathbb{C}}),
 {\mathfrak {gl}}(n, {\mathbb{C}}))$
 or $({\mathfrak {so}}(n+1, {\mathbb{C}}),
 {\mathfrak {so}}(n, {\mathbb{C}}))$
 up to outer automorphisms.  
\end{enumerate}
\end{prop}

\begin{proof}
The equivalence (ii) $\Leftrightarrow$ (iii)
 was proved
 by Kr{\"a}mer \cite{Kr}. 
Since any minimal parabolic subgroup
 is a Borel subgroup
 for complex reductive Lie groups,
 the equivalence (i) $\Leftrightarrow$ (ii)
 is obvious. 
\end{proof}
Alternatively, 
 the proof of Proposition \ref{prop:cpx-1}
 is covered by special cases
 of our propositions
 in later sections:
\begin{alignat*}{2}
&\text{Proposition \ref{prop:somn}}
\qquad&& ({\mathfrak {o}}(m+n,{\mathbb{C}}), {\mathfrak {o}}(m,{\mathbb{C}})+{\mathfrak {o}}(n,{\mathbb{C}})).  
\\
&\text{Proposition \ref{prop:rankH}}
&& \operatorname{rank}_{\mathbb{R}} H=1. 
\\
&\text{Proposition \ref{prop:nonKe}}
&& \operatorname{rank}_{\mathbb{R}} H \ge 2
\text{ or }
({\mathfrak {g}}, {\mathfrak {h}})
 \not \simeq 
({\mathfrak {o}}(m+n,{\mathbb{C}}), {\mathfrak {o}}(m,{\mathbb{C}})+{\mathfrak {o}}(n,{\mathbb{C}})).  
\end{alignat*}
See also Propositions \ref{prop:glgl}
 and \ref{prop:sp}
 for an alternative and direct proof
 for the pairs
 $({\mathfrak {sl}}(m+n,{\mathbb{C}}),
 {\mathfrak {s}}({\mathfrak {gl}}(m,{\mathbb{C}})+{\mathfrak {gl}}(n,{\mathbb{C}})))$
 and $({\mathfrak {sp}}(m+n,{\mathbb{C}}), {\mathfrak {sp}}(m,{\mathbb{C}})+{\mathfrak {sp}}(n,{\mathbb{C}}))$, 
respectively.  
\begin{remark}
\label{rem:cpx}
{\rm{
In \cite{Cooper, Kr}, 
 the pair $({\mathfrak {so}}(8,{\mathbb{C}}), {\mathfrak {spin}}(7,{\mathbb{C}}))$
 also appears in the classification.  
However,
 it is isomorphic to the pair
 $({\mathfrak {so}}(8,{\mathbb{C}}),
 {\mathfrak {so}}(7,{\mathbb{C}}))$
 by an outer automorphism
 of ${\mathfrak {so}}(8,{\mathbb{C}})$.  
The automorphism arises from the triality of $D_4$
 (see also Lemma \ref{lem:SS}).  
}}
\end{remark}

Since the condition (BB) is determined by the complexification
 of the pair $({\mathfrak {g}}, {\mathfrak {h}})$, 
 we have:
\begin{prop}
\label{prop:cpx}
Let $({\mathfrak {g}}, {\mathfrak {h}})$ 
 be an irreducible symmetric pair.  
Then the following two conditions 
 are equivalent:
\begin{enumerate}
\item[{\rm{(i)}}]
$({\mathfrak {g}}, {\mathfrak {h}})$ satisfies {\rm{(BB)}}.  
\item[{\rm{(ii)}}]
$({\mathfrak {g}}, {\mathfrak {h}})$
 is isomorphic to {\rm{(F1)}}-- {\rm{(F5)}}.  
\end{enumerate}
\end{prop}

\begin{remark}
{\rm{
The \lq\lq{group case}\rq\rq\
 $({\mathfrak {g}}'+{\mathfrak {g}}', \operatorname{diag}{\mathfrak {g}}')$
 satisfying {\rm{(BB)}} is 
 either
 ${\mathfrak {g}}' \simeq {\mathfrak {sl}}
  (2,{\mathbb{R}})$
 or 
 ${\mathfrak {g}}' \simeq {\mathfrak {sl}}
  (2,{\mathbb{C}})$
 when ${\mathfrak {g}}'$ is simple.  
They are included
 as special cases of {\rm{(F2)}} and {\rm{(F5)}}:
\begin{align*}
({\mathfrak{sl}}(2,{\mathbb{R}})+{\mathfrak{sl}}(2,{\mathbb{R}}), 
\operatorname{diag}{\mathfrak{sl}}(2,{\mathbb{R}}))
\approx&({\mathfrak{o}}(2,2), {\mathfrak{o}}(2,1))
\\
\approx&
({\mathfrak{o}}(2,1)+{\mathfrak{o}}(2,1), 
\operatorname{diag}{\mathfrak{o}}(2,1)).
\end{align*}
\begin{align*}
({\mathfrak{sl}}(2,{\mathbb{C}})+{\mathfrak{sl}}(2,{\mathbb{C}}), 
\operatorname{diag}{\mathfrak{sl}}(2,{\mathbb{C}}))
\simeq
&({\mathfrak{so}}(4,{\mathbb{C}}), {\mathfrak{so}}(3,{\mathbb{C}}))
\\
\simeq
&
({\mathfrak{o}}(3,1)+{\mathfrak{o}}(3,1), 
\operatorname{diag}{\mathfrak{o}}(3,1)).
\end{align*}
}}
\end{remark}
\section{Some classical and exceptional cases}
\label{sec:classical}

In this section,
 we deal with some classical symmetric pairs
 $(G,H)$ in matrix forms 
 and one exceptional symmetric pair.  
Classical symmetric spaces have parameters
 such as $i$, $j$, $k$ and $l$
 in $(G,H)=(O(i+j,k+l),O(i,k) \times O(j,l))$.  
We determine
 for which parameters they satisfy {\rm{(PP)}} or {\rm{(QP)}}
 by using the criteria, 
Theorems \ref{thm:pp} and \ref{thm:qp}.  
The cases we treat in Section \ref{sec:classical}
 cover (C)--(H)
 in Theorem \ref{thm:1.1}
 except for (E4)
 (see Step 4 in Section \ref{sec:strategy}).  
In particular,
 the implication (ii) $\Rightarrow$ (i)
 in Theorem \ref{thm:1.2} is proved in this section
 except for (E4).  
The case (E4) will be treated in Section \ref{sec:rank1}
 together with other symmetric pairs
 $({\mathfrak {g}}, {\mathfrak {h}})$
 with $\operatorname{rank}_{\mathbb{R}}{\mathfrak {h}}=1$.  

\subsection
{$(G,H)=(U(i+j, k+l;{\mathbb{F}}),
U(i, k;{\mathbb{F}}) \times U(j,l;{\mathbb{F}}))$}
\label{subsec:upq}
The open orbit properties (BB), (PP), and (QP)
 do not change
 if we replace $(G,H)$ 
 by their coverings, 
 connected components,
 or the quotients $(G/Z, H/H \cap Z)$
 by a central subgroup $Z$ of $G$.  
Thus, 
we shall treat the disconnected group 
 $O(p,q)$ rather than the connected group
 $SO_0(p,q)$, 
 and the reductive group $U(p,q)$ rather
 than the semisimple group $SU(p,q)$.  

The goal of this subsection
 is to prove the following:
\begin{prop}
\label{prop:upq}
Let $(G,H):=(U(i+j, k+l;{\mathbb{F}}),
U(i, k;{\mathbb{F}}) \times U(j,l;{\mathbb{F}}))$ 
with ${\mathbb{F}}={\mathbb{R}}$, 
${\mathbb{C}}$ or the quarternionic number field ${\mathbb{H}}$.  
Suppose that 
$l \le \min\{i,j,k\}$.  
\par\noindent
{\rm{1)}}\enspace
The pair $(G,H)$ satisfies
 {\rm{(QP)}}
 if and only if
\[
  l =0 \quad\text{ and }\,\, \min(i,j,k)=1.   
\]
\par\noindent
{\rm{2)}}\enspace
The pair $(G,H)$ satisfies
 {\rm{(PP)}}
 if and only if
\[
  l =0 \quad\text{ and }\,\, \min(j,k)=1.   
\]
\end{prop}
In particular,
 Proposition \ref{prop:upq} proves 
 the implication (ii) $\Rightarrow$ (i)
 in Theorem \ref{thm:1.1}
 for (E1), (E2), (E3), 
 (F4), (F5), and (H4).

In order to prove Proposition \ref{prop:upq}, 
we begin with the following:
\begin{lemma}
\label{lem:upql}
If $(G, H)$ satisfies {\rm{(QP)}}, 
 then $l=0$.  
\end{lemma}

\begin{proof}
By Theorem \ref{thm:qp}, 
 (QP) is equivalent
 to the existence
 of an open orbit of $M_H A_H$
 on ${\mathfrak {n}}^{-\sigma}$.  
Then the idea of the proof
 of the lemma 
 is to find 
a non-trivial $M_H A_H$-invariant function 
 on ${\mathfrak {n}}^{-\sigma}$
 for $l >0$.  
By a simple matrix computation,
 we have natural isomorphisms
 of groups and vector spaces:
\begin{align*}
M_H A_H
\simeq&
({\mathbb{F}}^{\times})^{\min(i,k)}
 \times U(|i-k|,{\mathbb{F}}) \times({\mathbb{F}}^{\times})^l 
\times U(j-l,{\mathbb{F}}), 
\\
{\mathfrak {n}}^{-\sigma}
\simeq &
M(i,l;{\mathbb{F}}) \oplus M(j,\min(i,k);{\mathbb{F}}).  
\end{align*}
Via these identifications,
 the adjoint action of an element
 $(a,A, b, B) \in M_H A_H$
 on the vector space ${\mathfrak {n}}^{-\sigma}$
 is given as 
\begin{equation}
\label{eqn:MAI}
(X,Y)
\mapsto
\begin{cases}
(
\begin{pmatrix}
a & \\
  & A
\end{pmatrix}
Xb^{-1}, 
\begin{pmatrix}
b & \\
  & B
\end{pmatrix}
Y a^{-1}
)
&\text{for }
i \ge k,
\\
(a X b^{-1}, \begin{pmatrix}
b & \\
  & B
\end{pmatrix}
Y a^{-1})
&\text{for }
i \le k.  
\end{cases}
\end{equation}
In the above formula,
 we have identified 
 $b \in ({\mathbb{F}}^{\times})^l$
 with a diagonal matrix
 in $GL(l,{\mathbb{F}})$, 
 and likewise for $a \in ({\mathbb{F}}^{\times})^{\min(i,k)}$.  

According to the partition 
 $j=l+(j-l)$, 
 we write $Y \in M(j,\min(i,k);{\mathbb{F}})$
 as a block matrix $Y=\begin{pmatrix} Y' \\ \ast \end{pmatrix}$
 with $Y' \in M(l,\min(i,k);{\mathbb{F}})$.

Consider the matrix
 $XY' \in M(i,\min(i,k);{\mathbb{F}})$.  
In view of the formula \eqref{eqn:MAI}
 of the $M_H A_H$-action 
 on ${\mathfrak {n}}^{-\sigma}$, 
 the $(1,1)$-component of $XY'$, 
 to be denoted by $z$, 
 is transformed as 
\[
  z \mapsto a_1 z a_1^{-1}, 
\]
where $a_1 \in {\mathbb{F}}^{\times}$
 is the first component of $a$.  
This formula means that the real algebraic function 
\[
     \psi:{\mathfrak {n}}^{-\sigma} \to {\mathbb{R}}, 
     \quad
     (X,Y) \mapsto |z|^2 
\] 
 is $M_H A_H$-invariant.  
Clearly, 
 $\psi$ is non-zero if $l>0$.  
Hence $M_H A_H$ cannot have an open orbit
 in ${\mathfrak {n}}^{-\sigma}$
 if $l >0$.  
\end{proof}

To complete the proof
 of Proposition \ref{prop:upq}
 we need the following elementary lemma:
\begin{lemma}
\label{lem:FGauss}
Suppose $0 \le p' \le p$
 and $p,q \ge 1$.  
We let $S:= U(1, {\mathbb{F}})^{p'}
 \times U(p-p',{\mathbb{F}}) \times ({\mathbb{F}}^{\times})^q$
 act on $M(p,q;{\mathbb{F}})$ by 
\[
   X \mapsto \begin{pmatrix} a & \\ & A\end{pmatrix} X b^{-1}
\quad
\text{for }
\,\,
(a,A,b) \in L
\text{ and }
X \in M(p,q;{\mathbb{F}}).  
\]
Then the group $S$ has an open orbit 
 in $M(p,q;{\mathbb{F}})$
 if and only if 
 $p=1$ or ($q=1$ and $p'=0$).  
\end{lemma}
\begin{proof}
First we observe
 that $({\mathbb{F}}^{\times})^q$ acts 
 on the quotient space
 $U(p,{\mathbb{F}}) \backslash M(p,q;{\mathbb{F}})$
{}from the right.  
This action has an open orbit 
 if and only if 
 $p=1$ or $q=1$
 by the Gauss decomposition 
 for $U(p,{\mathbb{F}}) \backslash M(p,q;{\mathbb{F}})$.  
Hence the equivalence assertion of Lemma \ref{lem:FGauss}
 is proved for $p'=0$.  
What remains to prove
 is that there is no open orbit
 if $p'>0$ and $q=1$, 
 but this is obvious.  
\end{proof}

\begin{proof}
[Proof of Proposition \ref{prop:upq}]
1)\enspace
By Lemma \ref{lem:upql}, 
we may and do assume $l=0$.  
Then the group 
\[
  M_H A_H \simeq ({\mathbb{F}}^{\times})^{\min(i,k)}
                 \times U(|i-k|,{\mathbb{F}}) \times U(j,{\mathbb{F}})
  \ni (a,A,B)
\]
acts on ${\mathfrak{n}}^{-\sigma} \simeq M(j,\min(i,k);{\mathbb{F}})$
by $Y \mapsto BYa^{-1}$.  
We observe 
 that the second factor
 $U(|i-k|, {\mathbb{F}})$
 of $M_H A_H$ acts trivially
 on ${\mathfrak{n}}^{-\sigma}$.  
Then by the criterion in Theorem \ref{thm:qp}, 
 the first statement follows as a special case
 of Lemma \ref{lem:FGauss}
 with $p=j$ and $q=\min(i,k)$.  
\par\noindent
2)\enspace
We need to prove
 that (PP) holds
 if $l=0$ and $\min(j,k)=1$, 
 and fails if $l=0$, $=1$
 and $\min(j,k)>1$.  
We shall use Theorem \ref{thm:pp}
 and Lemma \ref{lem:FGauss}.  
\par\noindent
{\bf{Case 2-1.}}\enspace
$k=1$.  
Then 
$
   \operatorname{rank}_{{\mathbb{R}}}H
   =
   \operatorname{rank}_{{\mathbb{R}}}G$
 $(=1)$, 
 and therefore (QP) is equivalent to (PP)
 by Lemma \ref{lem:BPQ} (2).  
Hence (PP) holds.  
\par\noindent
{\bf{Case 2-2.}}\enspace
$j=1$.  
We apply Lemma \ref{lem:FGauss}
 with $p=1$ and $q=\min(i,k)$
 to conclude 
 that the action of $(M_H \cap M_G) A_H$
 has an open orbit in 
$
   {\mathfrak {n}}^{-\sigma} \simeq {\mathbb{F}}^{\min(i,k)}
$.  
\par\noindent
{\bf{Case 2-3.}}\enspace
$i=1$.  
We apply Lemma \ref{lem:FGauss}
 with $p=j$, $p'=\min(j,k-1)$
 and $q=1$
 to conclude that the action of $(M_H \cap M_G) A_H$
 does not have an open orbit if $j,k >1$.  
\par
Hence the proof of Proposition \ref{prop:upq}
 is completed.  
\end{proof}

\subsection{$(G,H)=(GL(p+q,{\mathbb{F}}),
 GL(p,{\mathbb{F}}) \times GL(q,{\mathbb{F}}))$}
\label{subsec:GLGL}
Next we treat
 the symmetric pair
 $(G,H)=(GL(p+q,{\mathbb{F}}),
 GL(p,{\mathbb{F}}) \times GL(q,{\mathbb{F}}))$
 for ${\mathbb{F}}={\mathbb{R}}$, 
 ${\mathbb{C}}$ and ${\mathbb{H}}$.  
Surprisingly,
 the property (PP) behaves
 uniformly
 for all ${\mathbb{F}}={\mathbb{R}}$, 
 ${\mathbb{C}}$ and ${\mathbb{H}}$ for this pair.  
In contrast, 
 that the property (BB)
 behaves differently 
 for ${\mathbb{F}}={\mathbb{H}}$
 (see Remark \ref{rem:GLGL} below).  
\begin{prop}
\label{prop:glgl}
Let $p,q \ge 1$ and 
\[
   (G,H)=(GL(p+q,{\mathbb{F}}), GL(p,{\mathbb{F}}) \times GL(q,{\mathbb{F}})), 
\quad
{\mathbb{F}}={\mathbb{R}}, \,{\mathbb{C}}\text{ or }{\mathbb{H}}.  
\]
Then the following three conditions are equivalent:
\begin{enumerate}
\item[{\rm{(i)}}]
The pair $(G,H)$ satisfies {\rm{(QP)}}.  
\item[{\rm{(ii)}}]
The pair $(G,H)$ satisfies {\rm{(PP)}}.  
\item[{\rm{(iii)}}]
$\min(p,q)=1$.  
\end{enumerate}
\end{prop}
\begin{remark}
\label{rem:GLGL}
{\rm{
For $\min(p,q)=1$, 
 $(G,H)$ satisfies (BB)
 if and only if ${\mathbb{F}}={\mathbb{R}}$ or ${\mathbb{C}}$.  
In fact,
 for ${\mathbb{F}}={\mathbb{H}}$, 
 the complexified Lie algebra
\[
({\mathfrak {g}}_{\mathbb{C}}, 
   {\mathfrak {h}}_{\mathbb{C}})
 \simeq
  ({\mathfrak {gl}}(2(p+q),{\mathbb{C}}), 
   {\mathfrak {gl}}(2p,{\mathbb{C}}) 
  + {\mathfrak {gl}}(2q,{\mathbb{C}}))
\]
 cannot be a strong Gelfand pair
 (see Proposition \ref{prop:cpx-1}).  
}}
\end{remark}
\begin{proof}
[Proof of Proposition \ref{prop:glgl}]
Via the isomorphisms
\[
M_H A_H \simeq ({\mathbb{F}}^{\times})^{p+q}, 
\quad
{\mathfrak {n}}^{-\sigma} \simeq M(p,q;{\mathbb{F}}),
\]
the adjoint action
of $M_H A_H$ on ${\mathfrak {n}}^{-\sigma}$
 is given as the action
of $({\mathbb{F}}^{\times})^{p} \times ({\mathbb{F}}^{\times})^{q}$
 on $M(p,q;{\mathbb{F}})$
 by the left and right multiplication.

If $p,q \ge 2$, 
then 
\[
M(p,q;{\mathbb{F}}) \to {\mathbb{R}}, 
\quad
X \mapsto |X_{11}X_{22}|^2/|X_{12}X_{21}|^2
\]
is well-defined
 on an open dense subset
 of $M(p,q;{\mathbb{F}})$
 and is invariant 
 by the action of 
$
   ({\mathbb{F}}^{\times})^{p} \times ({\mathbb{F}}^{\times})^{q}, 
$
 and thus there is no open orbit.

Conversely,
 if $q=1$, 
 then clearly $({\mathbb{F}}^{\times})^{p}$
 has an open orbit in ${\mathbb{F}}^{p}$, 
 and so does $M_H A_H$ in ${\mathfrak {n}}^{-\sigma}$.  
Thus the equivalence (i) $\Leftrightarrow$ (iii)
 follows from Theorem \ref{thm:qp}.

Since $\operatorname{rank}_{\mathbb{R}}H
=\operatorname{rank}_{\mathbb{R}}G$, 
 the equivalence (i) $\Leftrightarrow$ (ii) holds.  
Hence Proposition is proved.  
\end{proof}

\subsection{$(G,H)=(O(m+n,{\mathbb{C}}),O(m,{\mathbb{C}}) \times O(n,{\mathbb{C}}))$}
\label{subsec:OmnC}
The main goal of this section
 is to prove the following proposition.  
The equivalence (i) $\Leftrightarrow$ (ii) $\Leftrightarrow$ (iii)
 is a special case of Proposition \ref{prop:cpx-1}.  
We shall use Proposition \ref{prop:somn}
 in Proposition \ref{prop:nonKe}.  
\begin{prop}
\label{prop:somn}
Let 
\[
(G,H)=(O(m+n,{\mathbb{C}}), O(m,{\mathbb{C}}) \times O(n,{\mathbb{C}}))
\text{ with }
m, n \ge 1.  
\]
Then the following four conditions
 are equivalent:
\par\noindent
{\rm{(i)}}\enspace
The pair $(G,H)$ satisfies {\rm{(QP)}}.  
\par\noindent
{\rm{(ii)}}\enspace
The pair $(G,H)$ satisfies {\rm{(PP)}}.  
\par\noindent
{\rm{(iii)}}\enspace
$m=1$, $n=1$ or $(m,n)=(2,2)$.  
\par\noindent
{\rm{(iv)}}\enspace
$\operatorname{rank}_{\mathbb{R}}
H \ge \# \Delta({\mathfrak {n}}^{-\sigma})$.  
\end{prop}
\begin{proof}
The equivalence (iii) $\Leftrightarrow$ (iv):
{}From the table below,
 we see that $({\mathfrak {g}}, {\mathfrak {h}})$
 satisfies (iv)
 if and only if
 $m=1$, $n=1$, or $(m,n)=(1,2)$.  
\begin{alignat*}{4}
&\quad m
&&\quad n 
&& \operatorname{rank}_{\mathbb{R}} H
\qquad
&& \# \Delta({\mathfrak {n}}^{-\sigma})
\\
&  2p+1
\qquad
&& 2q+1
\qquad
&&\, p+q
&& 2pq + p+q
\\
&  2p+1
&& 2 q
&&\, p+q
&& 2pq + q
\\
&  2p
&& 2q
&&\, p+q
&& 2pq
\end{alignat*}
\end{proof}
\subsection{$(G,H)=(O^{\ast}(2p+2q), O^{\ast}(2p) \times O^{\ast}(2q))$}
The goal of this subsection
 is to prove the following:
\begin{prop}
\label{prop:sostar}
Let \[
(G,H)=(O^{\ast}(2p+2q), O^{\ast}(2p) \times O^{\ast}(2q))
\text{ with }
p,q \ge 1.  
\]
The following three conditions
 are equivalent:
\par\noindent
{\rm{(i)}}\enspace
The pair $(G,H)$ satisfies {\rm{(QP)}}.  
\par\noindent
{\rm{(ii)}}\enspace
The pair $(G,H)$ satisfies {\rm{(PP)}}.  
\par\noindent
{\rm{(iii)}}\enspace
$\min(p,q)=1$.  
\end{prop}
In particular,
 Proposition \ref{prop:sostar} shows the implication 
 (ii) $\Rightarrow$ (i) in Theorem \ref{thm:1.1}
 for (H3).  

In order to give a proof
 of the implication (i) $\Rightarrow$ (iii) 
 in Proposition \ref{prop:sostar},  
 we use Proposition \ref{prop:QPrank}.  
For this,
 we need:

\begin{lemma}
\label{lem:sostar}
$\operatorname{rank}_{\mathbb{R}}
H \ge \# \Delta({\mathfrak {n}}^{-\sigma})$
if and only if 
$\min(p,q)=1$
 or $(p,q)=(2,2)$.  
\end{lemma}
\begin{proof}
Lemma \ref{lem:sostar} is an immediate consequence
 of the formulae:
\[
\operatorname{rank}_{\mathbb{R}}{\mathfrak {a}}_H
=[\frac p 2]+[\frac q 2]
\quad
\text{ and } 
\#\Delta({\mathfrak {n}}^{-\sigma})=[\frac {pq} 2].  
\]
\end{proof}

\begin{proof}[Proof of Proposition \ref{prop:sostar}]
The implication (ii) $\Rightarrow$ (i) holds
 in general by Lemma \ref{lem:BPQ} (1).  
By Lemma \ref{lem:sostar} and Proposition \ref{prop:QPrank}, 
 (QP) holds
 only if $\min(p,q)=1$
 or $(p,q)=(2,2)$.  
What remains to prove is:
\begin{enumerate}
\item[$\bullet$]
(QP) fails if $(p,q)=(2,2)$. 
\item[$\bullet$]
(PP) holds if $q=1$. 
\end{enumerate}
In view of the isomorphism 
 of symmetric pairs:
\[
({\mathfrak {o}}^{\ast}(8), 
{\mathfrak {o}}^{\ast}(4)+{\mathfrak {o}}^{\ast}(4))
\simeq
({\mathfrak {o}}(2,6), {\mathfrak {o}}(2,2)+{\mathfrak {o}}(4)), 
\]
the first claim is regarded as a special case
 of Proposition \ref{prop:upq}, 
 which we have already proved.

To see the second claim,
suppose $q=1$.  
Then we have the following natural isomorphisms
 of groups and vector spaces,
 respectively:
\begin{align*}
M_H A_H
\simeq&
\begin{cases}
({\mathbb{H}}^{\times})^{\frac{p}{2}} \times {\mathbb{T}}
\quad
&\text{($p$:even)}, 
\\
({\mathbb{H}}^{\times})^{\frac{p-1}{2}} \times {\mathbb{T}}^2
\quad
&\text{($p$:odd)}, 
\end{cases}
\\
(M_H \cap M_G)A_H
\simeq&
({\mathbb{H}}^{\times})^{[\frac{p}{2}]} \times {\mathbb{T}}, 
\\
 {\mathfrak {n}}^{-\sigma}
\simeq & {\mathbb{H}}^{[\frac p 2]}.   
\end{align*}
Via these isomorphisms,
 the adjoint action
of the first factor
 of $(M_H \cap M_G)A_H$ on ${\mathfrak {n}}^{-\sigma}$
 is given by the natural action 
 of $({\mathbb{H}}^{\times})^{[\frac{p}{2}]}$
 on ${\mathbb{H}}^{[\frac{p}{2}]}$, 
which has obviously an open dense orbit.  
By Theorem \ref{thm:pp}, 
 we conclude 
 that (PP) holds if $q=1$.  

Hence the proof of Proposition \ref{prop:sostar}
 is completed.  
\end{proof}

\subsection{$(G,H)$ is of type $(C_n, A_n)$}
\label{subsec:CA}
In this subsection,
 we treat the reductive symmetric pairs
 $({\mathfrak {g}}, {\mathfrak {h}})$
 that have the following three properties:

\begin{align}
&\text
{
The root system $\Sigma({\mathfrak {g}}, {\mathfrak {a}}_H)$
 is of type $C_n$, 
}
\label{eqn:CA1}
\\
&\text{
The root system 
 $\Sigma({\mathfrak {h}}, {\mathfrak {a}}_H)$
 is of type $A_n$, 
}
\label{eqn:CA2} 
\\
&\text{
Either $m^+(\lambda)=0$
 of $m^-(\lambda)=0$
  for each $\lambda \in \Sigma({\mathfrak {g}}, {\mathfrak {a}}_H)$.}
\label{eqn:CA3}
\end{align}
Here we define
\[
  m^{\pm}(\lambda)= \dim_{\mathbb{R}} 
  {\mathfrak {g}}^{\pm\sigma}({\mathfrak {a}}_H;\lambda).  
\]
The main results of this subsection are 
Propositions \ref{prop:ugl} and \ref{prop:ug2}.  

\begin{prop}
\label{prop:ugl}
Let $(G,H)$ be one of the following symmetric pairs:
\begin{align*}
&(U(n,n;{\mathbb{F}}),GL(n,{\mathbb{F}}))
\quad
{\mathbb{F}}={\mathbb{C}}\text{ or }{\mathbb{H}}, 
\\
&(Sp(n,{\mathbb{R}}),GL(n,{\mathbb{R}})), 
\\
&(O^{\ast}(4n),GL(n,{\mathbb{H}})).  
\end{align*}
Then the following three conditions are equivalent:
\begin{enumerate}
\item[{\rm{(i)}}]
The pair $(G,H)$ satisfies {\rm{(QP)}}.  
\item[{\rm{(ii)}}]
The pair $(G,H)$ satisfies {\rm{(PP)}}.  
\item[{\rm{(iii)}}]
$n=1$.  
\end{enumerate}
\end{prop}
To begin with, 
we observe the following:
\begin{lemma}
\label{lem:CA}
The three families
 of symmetric pairs 
 $({\mathfrak {g}}, {\mathfrak {h}})$ 
 in Proposition \ref{prop:ugl}
 satisfy \eqref{eqn:CA1}, 
  \eqref{eqn:CA2}, and  \eqref{eqn:CA3}.  
\end{lemma}
\begin{proof}
We take the standard basis
 $\{e_1, \cdots, e_n\}$
 of ${\mathfrak {a}}_H^{\ast}$
 such that 
 $\Sigma({\mathfrak {h}}, {\mathfrak {a}}_H)
 =\{\pm(e_i-e_j):
    1\le i < j \le n\}$.  
We set $d=\dim_{\mathbb{R}} {\mathbb{F}}$.  
Then the pair of multiplicities
$\begin{pmatrix} m^+(\lambda)\\ m^-(\lambda)\end{pmatrix}$
is given as follows:
\begin{table}[H]
\centering
\begin{tabular}{cccccc}
& ${\mathfrak {g}}$
& ${\mathfrak {h}}$
& $e_i-e_j$
& $e_i+e_j$
& $\quad 2 e_l$
\\
\\
& ${\mathfrak {u}}(n,n;{\mathbb{F}})$\qquad
& ${\mathfrak {gl}}(n,{\mathbb{F}})$\qquad
& $\begin{pmatrix} d \\ 0\end{pmatrix}$\qquad
& $\begin{pmatrix} 0 \\ d\end{pmatrix}$\qquad
& $\begin{pmatrix} 0 \\ d-1 \end{pmatrix}$\qquad
\\
\\
& ${\mathfrak {sp}}(n,{\mathbb{R}})$
& ${\mathfrak {gl}}(n,{\mathbb{R}})$
& $\begin{pmatrix} 1 \\ 0\end{pmatrix}$
& $\begin{pmatrix} 0 \\ 1\end{pmatrix}$
& $\quad\begin{pmatrix} 0 \\ 1\end{pmatrix}$
\\
\\
& ${\mathfrak {so}}^{\ast}(4n)$
& ${\mathfrak {gl}}(n,{\mathbb{H}})$
& $\begin{pmatrix} 4 \\ 0\end{pmatrix}$
& $\begin{pmatrix} 0 \\ 4\end{pmatrix}$
& $\quad\begin{pmatrix} 0 \\ 1\end{pmatrix}$
\end{tabular}
\caption{$(m^+(\lambda), m^-(\lambda))$
 for symmetric pairs
 of type $(C_n,A_n)$}
\label{tab:4.1}
\end{table}
The lemma is clear from Table \ref{tab:4.1}.  
\end{proof}

\begin{proof}
[Proof of Proposition \ref{prop:ugl}]
By \eqref{eqn:CA1}, \eqref{eqn:CA2}
 and \eqref{eqn:CA3}, 
$\#\Delta({\mathfrak {n}}^{-\sigma})$ is equal
 to half the difference of the cardinalities of roots
 in $C_n$ and $A_n$, 
 namely,
\[
  \#\Delta({\mathfrak {n}}^{-\sigma})
  =
  \frac 1 2 (2n^2-(n^2-n))
  =\frac 1 2 n(n+1).  
\]
Therefore, 
 the inequality 
 $\operatorname{rank}_{\mathbb{R}} H \ge \#\Delta({\mathfrak {n}}^{-\sigma})$
 amounts to $n \ge \frac 1 2 n (n+1)$, 
 namely, 
 $n=1$.  
By Proposition \ref{prop:QPrank}, 
 we have shown the implication (i) $\Rightarrow$ (iii).

The equivalence (i) $\Leftrightarrow$ (ii)
 follows from Lemma \ref{lem:BPQ} (2)
 because $\operatorname{rank}_{\mathbb{R}} H
=\operatorname{rank}_{\mathbb{R}} G$.

Finally, 
 the implication (iii) $\Rightarrow$ (i)
 is included
 in a special 
 (and easy)
 case of other families, 
which we have already shown to 
 satisfy {\rm{(PP)}}.  
In fact,
\begin{align*}
({\mathfrak {u}}(1,1), {\mathfrak {gl}}(1,{\mathbb{C}}))
 \simeq\,& 
({\mathfrak {o}}(2,1), {\mathfrak {o}}(1,1))
+({\mathbb{R}}, {\mathbb{R}}), 
\\
({\mathfrak {sp}}(1,1), {\mathfrak {gl}}(1,{\mathbb{H}}))
 \simeq\,& ({\mathfrak {o}}(1,4), {\mathfrak {o}}(1,1)+{\mathfrak {o}}(3)), 
\\
({\mathfrak {sp}}(1,{\mathbb{R}}), {\mathfrak {gl}}(1,{\mathbb{R}}))
 \simeq\,& ({\mathfrak {o}}(2,1), {\mathfrak {o}}(1,1)), 
\\
({\mathfrak {o}}^{\ast}(4), {\mathfrak {gl}}(1,{\mathbb{H}}))
 \simeq\,& ({\mathfrak {o}}(3),{\mathfrak {o}}(3))
\oplus 
({\mathfrak {o}}(1,2),{\mathfrak {o}}(1,1)).  
\end{align*}
We know
 that the symmetric pairs
 in the right-hand side satisfy (PP) as special cases
 of Proposition \ref{prop:upq}.  
Thus we have proved Proposition \ref{prop:ugl}. 
\end{proof}
We end this subsection 
 with the symmetric pair
 $(U(n,n;{\mathbb{F}}), GL(n,{\mathbb{F}}))$ 
 for ${\mathbb{F}}={\mathbb{R}}$, 
which was excluded from Proposition \ref{prop:ugl}
 as a \lq\lq{degenerate case}\rq\rq.  
\begin{prop}
\label{prop:ug2}
Let $(G,H)=(O(n,n), GL(n,{\mathbb{R}}))$
 $(n \ge 2)$.  
Then {\rm{(QP)}}$\Leftrightarrow$
{\rm{(PP)}}$\Leftrightarrow$ $n=2$ or $3$.  
\end{prop}
\begin{proof}
The root multiplicities are given 
 in the first row of Table \ref{tab:4.1}.  
We observe
 that the long roots $\pm 2 e_l$ do not appear
 because $d=1$
 for ${\mathbb{F}}={\mathbb{R}}$.  
As a result,
 we have $\#\Delta({\mathfrak {n}}^{-\sigma})
=\frac 1 2 n(n-1)$,
 and the inequality $\operatorname{rank}_{\mathbb{R}}H \ge 
\#\Delta({\mathfrak {n}}^{-\sigma})$
 amounts to $n\ge \frac 1 2 n(n-1)$, 
namely,
 $n=2$ or $3$.  
Thus we have proved 
 the implications
 (PP) $\Leftrightarrow$(QP) $\Rightarrow$
 $n=2$ or $3$
 by Proposition \ref{prop:QPrank}.  

Conversely,
 for $n=2$, $3$, 
 we observe the following isomorphisms:
\begin{align*}
({\mathfrak {o}}(2,2), {\mathfrak {gl}}(2,{\mathbb{R}}))
\simeq\, &
({\mathfrak {o}}(1,2),{\mathfrak {o}}(1,2)) \oplus 
({\mathfrak {o}}(1,2),{\mathfrak {o}}(1,1)),
\\
({\mathfrak {o}}(3,3), {\mathfrak {gl}}(3,{\mathbb{R}}))
\simeq\, &
({\mathfrak {sl}}(4,{\mathbb{R}}),{\mathfrak {gl}}(3,{\mathbb{R}})).  
\end{align*}
They satisfy (PP) 
 as special cases of Propositions \ref{prop:upq}
 and \ref{prop:cpx}, 
 respectively.  
\end{proof}

\subsection{$(G,H)=(Sp(p+q,{\mathbb{F}}),
 Sp(p,{\mathbb{F}})\times Sp(q,{\mathbb{F}}))$, 
 ${\mathbb{F}}={\mathbb{R}}$ or ${\mathbb{C}}$}
\label{subsec:Sp}
\begin{prop}
\label{prop:sp}
Let $p, q \ge 1$
 and 
\[
   (G,H)=(Sp(p+q,{\mathbb{F}}),
 Sp(p,{\mathbb{F}}) \times Sp(q,{\mathbb{F}})), 
\qquad
 {\mathbb{F}}={\mathbb{R}}\,\,\text{ or }\,\, {\mathbb{C}}. 
\] 
Then {\rm{(QP)}} $\Leftrightarrow$
 {\rm{(PP)}} $\Leftrightarrow$ $(p,q)=(1,1)$.  
\end{prop}
\begin{proof}
Take the standard basis
$\{f_1, \cdots, f_{p+q}\}$
 of ${\mathfrak {a}}_H^{\ast}={\mathfrak {a}}_G^{\ast}$
 such that
\[
\Delta({\mathfrak {n}}^{-\sigma})
=
\{f_i \pm f_j: 1 \le i \le p, p+1 \le j \le p+q\}.  
\]
Then the inequality 
 $\operatorname{rank}_{\mathbb{R}}H \ge \# \Delta({\mathfrak {n}}^{-\sigma})$
 amounts to 
 $p+q \ge 2 pq$, 
 which holds only if $(p,q)=(1,1)$.  
Therefore,  
 if $(G,H)$ satisfies (QP), 
then $(p,q)=(1,1)$ by Theorem \ref{thm:qp}
 and Proposition \ref{prop:QPrank}.  
Conversely, 
 if $(p,q)=(1,1)$, 
then 
\begin{align*}
&({\mathfrak {sp}}(2,{\mathbb{R}}), 
{\mathfrak {sp}}(1,{\mathbb{R}})+{\mathfrak {sp}}(1,{\mathbb{R}}))
\simeq
 ({\mathfrak {o}}(3,2), {\mathfrak {o}}(2,2)),
\\
& ({\mathfrak {sp}}(2,{\mathbb{C}}),  
  {\mathfrak {sp}}(1,{\mathbb{C}})+{\mathfrak {sp}}(1,{\mathbb{C}}))
\simeq
 ({\mathfrak {o}}(5,{\mathbb{C}}), {\mathfrak {o}}(4,{\mathbb{C}})),
\end{align*}
which satisfy (PP)
 as we have seen in Propositions \ref{prop:upq}
 and \ref{prop:somn}, 
 respectively.  
Hence Proposition \ref{prop:sp} is proved.  
\end{proof}

\subsection{$(G,H)=(O^{\ast}(2p+2q), U(p,q))$}
\label{subsec:OU}
As a final example of classical symmetric pairs,
 we consider 
 $({\mathfrak {g}}, {\mathfrak {h}})
 =({\mathfrak {o}}^{\ast}(2p+2q), {\mathfrak {u}}(p,q))$
 which is the $c$-dual of the symmetric pair
 $({\mathfrak {o}}(2p+2q), {\mathfrak {u}}(p,q))$.  
\begin{prop}
\label{prop:OUpq}
Let 
\[
(G,H)=(O^{\ast}(2p+2q), U(p,q))
\quad
\text{with }
p \ge q \ge 1.  
\]
\begin{enumerate}
\item[{\rm{1)}}]
The pair $(G,H)$ satisfies {\rm{(QP)}}
 if and only if $q=1$.  
\item[{\rm{2)}}]
The pair $(G,H)$ satisfies {\rm{(PP)}}
 if and only if $(p,q)=(3,1), (2,1)$
 or $(1,1)$.  
\end{enumerate}
\end{prop}

\begin{proof}
We take the standard basis
 $\{e_1, \cdots, e_q\}$
 of ${\mathfrak {a}}_H^{\ast}$
such that
\[
\Sigma({\mathfrak {g}}, {\mathfrak {a}}_H)
 \subset \{\pm e_i \pm e_j: 1 \le i < j \le q\}
  \cup 
  \{\pm e_l, \pm 2e_l: 1 \le l \le q\}.  
\]
The inclusion is actually the equality
 if and only if $p>q$.  
Further,
 the root multiplicities 
 $m^{\pm}(\lambda)$
 are given according to the parity of $p+q$ as follows:

{\bf{Case 1.}}\enspace
$p \equiv q \mod 2$.  

\begin{equation*}
\begin{array}{ccccc}
&\lambda
&\,\,\pm e_i \pm e_j\,\,
&\pm e_l
&\pm 2 e_l
\\
&m^{+}(\lambda)
& 2 
&\,\,2(p-q)\,\,
& 1
\\
&m^{-}(\lambda)
&2
&\,\,2(p-q)\,\,
&0
\end{array}
\end{equation*}

{\bf{Case 2.}}\enspace
$p \equiv q+1 \mod 2$.  
\begin{equation*}
\begin{array}{ccccc}
&\lambda
&\,\,\pm e_i \pm e_j\,\,
&\pm e_l
&\pm 2 e_l
\\
&m^{+}(\lambda)
& 2 
&\,\,2(p-q+1)\,\,
& 1
\\
&m^{-}(\lambda)
&2
&2(p-q+1)
&0
\end{array}
\end{equation*}
Thus we can take a positive system 
such that
\[
\Delta({\mathfrak {n}}^{-\sigma})
=\{\pm e_i \pm e_j: 1 \le i < j \le q\}
\,\,
  (\cup 
  \{\pm e_l: 1 \le l \le q\}
  \,\,\text{ for }\,\, p>q).  
\]
Hence the inequality 
 $\operatorname{rank}_{\mathbb{R}}H
 \ge \# \Delta({\mathfrak {n}}^{-\sigma})$
 implies
\[
  q \ge
  \begin{cases}
  q(q-1)
  \qquad
  &\text{for }\,\, p=q,
\\
  q(q-1)+q
  \qquad
  &\text{for }\,\, p>q.  
  \end{cases}
\]
By Proposition \ref{prop:QPrank}, 
 if $(G,H)$ satisfies (QP)
 then $(p,q)=(2,2)$
 or $q=1$.  

Conversely,
 suppose that $(p,q)=(2,2)$.  
In view of the isomorphism
\[
 ({\mathfrak{o}}^{\ast}(8),{\mathfrak{u}}(2,2))
 \simeq 
 ({\mathfrak{o}}(6,2),{\mathfrak{o}}(4,2)+{\mathfrak{o}}(2)), 
\]
we see that $(G,H)$ does not satisfy (QP)
 by Proposition \ref{prop:upq} (1).

Suppose now that $q=1$.  
Then $\operatorname{rank}_{\mathbb{R}}H =1$, 
 and we shall show in Proposition \ref{prop:rankH}
 that $(G,H)$ satisfies (QP)
 for any $p$ and (PP) for $p \le 3$
 (see III in Table \ref{tab:5.2}).  
This completes the proof of Proposition \ref{prop:OUpq}.  
\end{proof}

\subsection
{$({\mathfrak {g}}, {\mathfrak {h}})=({\mathfrak {e}}_{6(-26)}, {\mathfrak {so}}(9,1)+{\mathbb{R}})$}
\label{subsec:e6}

The exceptional real Lie algebra
$
{\mathfrak {g}}:=
{\mathfrak {e}}_{6(-26)}
$
 is a simple Lie algebra
 with the following property:
\[
{\mathfrak {k}} \simeq {\mathfrak {f}}_{4(-52)}
\,\,
\text{and }
\,\,
\operatorname{rank}_{\mathbb{R}}{\mathfrak {g}}=2.  
\]
The goal of this subsection
 is to prove the following:
\begin{prop}
\label{prop:e6}
Let $(G,H)$ be a symmetric pair
 with Lie algebras
\[
({\mathfrak {g}}, {\mathfrak {h}})
=
({\mathfrak {e}}_{6(-26)}, {\mathfrak {so}}(9,1)+{\mathbb{R}}).  
\]
Then $(G,H)$ satisfies {\rm{(PP)}} and {\rm{(QP)}}.  
\end{prop}

We begin with the Lie algebra
 ${\mathfrak {g}}={\mathfrak {e}}_{6(-26)}$.  
Then the Lie algebra 
${\mathfrak {m}}_G=Z_{{\mathfrak {k}}}({\mathfrak {a}}_G)
\simeq {\mathfrak {so}}(8)$
 acts on ${\mathfrak {n}}\simeq {\mathbb{R}}^{24}$
 via the adjoint action
 as the direct sum of the following three non-isomorphic 
 8-dimensional irreducible representations:
\begin{alignat*}{2}
&
&&\text{highest weight}
\\
&\text{Natural representation $i$}
&&
\lambda_1=(1,0,0,0), 
\\
&\text{Half spin representation $\operatorname{spin}^+$}
\qquad
&&
\lambda_2=\frac 1 2 (1,1,1,1), 
\\
&\text{Half spin representation $\operatorname{spin}^-$}
&&
\lambda_3=\frac 1 2 (1,1,1,-1).   
\end{alignat*}
Here the highest weights
 are expressed
 by means of the standard basis of $D_4$
 as in the proof of Lemma \ref{lem:2.6}.  

These representations are the differentials 
 of the representations
 of $Spin(8)$, 
to be denoted by the same letters
 $i$, $\operatorname{spin}^+$, 
and $\operatorname{spin}^-$, 
 respectively,
 which in turn induce three actions
 on the 7-dimensional sphere
 $S^7 \simeq ({\mathbb{R}}^8 - \{0\}) /{\mathbb{R}}_{>0}$.  
We need the following:
\begin{lemma}
\label{lem:SS}
Let $Spin(8)$ act diagonally
 on the direct product manifold
 $S^7 \times S^7$
 via any choice
 of two distinct 8-dimensional representations
 among $i$, $\operatorname{spin}^+$, 
and $\operatorname{spin}^-$.  
Then the action is transitive.  
\end{lemma}
\begin{proof}
The automorphism of the Dynkin diagram $D_4$
 gives rise to the triality in
$Spin(8)$.
We denote by $\sigma$ the outer automorphism of $Spin(8)$
 of order three
corresponding to the outer automorphism of $D_4$ as described in the
figure below.
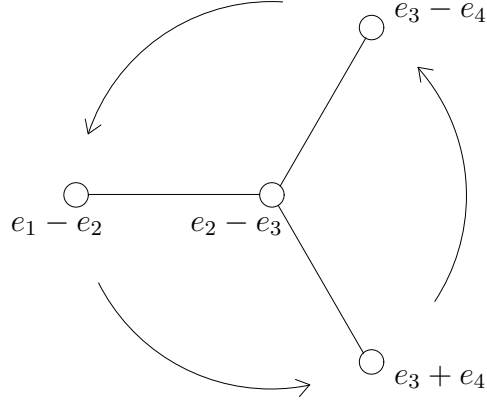
\begin{center}
\setlength{\unitlength}{0.00043333in}
\begingroup\makeatletter\ifx\SetFigFont\undefined%
\gdef\SetFigFont#1#2#3#4#5{%
  \reset@font\fontsize{#1}{#2pt}%
  \fontfamily{#3}\fontseries{#4}\fontshape{#5}%
  \selectfont}%
\fi\endgroup%
{\renewcommand{\dashlinestretch}{30}
\begin{picture}(5594,4885)(0,-10)
\put(3195.889,2494.083){\arc{4733.668}{3.4627}{4.7699}}
\put(3220.341,2483.281){\arc{4730.882}{5.5570}{6.8647}}
\put(3189.289,2473.087){\arc{4731.234}{1.3679}{2.6755}}
\put(800,2491){\ellipse{300}{300}}
\put(4420,441){\ellipse{300}{300}}
\put(4420,4541){\ellipse{300}{300}}
\put(3200,2491){\ellipse{300}{300}}
\path(950,2491)(3050,2491)
\path(4325,541)(3275,2360)
\path(3302,2594)(4352,4413)
\path(3538,12)(3659,155)
\path(5176,4017)(4992,4051)
\path(5024,3867)(4993,4051)
\path(948,3232)(914,3416)
\path(950,3236)(1104,3343)
\path(3664,157)(3489,222)
\put(4700,175){\makebox(0,0)[lb]{\smash{{{\SetFigFont{12}{14.4}{\familydefault}{\mddefault}{\updefault}$e_3+e_4$}}}}}
\put(4700,4666){\makebox(0,0)[lb]{\smash{{{\SetFigFont{12}{14.4}{\familydefault}{\mddefault}{\updefault}$e_3-e_4$}}}}}
\put(2200,2051){\makebox(0,0)[lb]{\smash{{{\SetFigFont{12}{14.4}{\familydefault}{\mddefault}{\updefault}$e_2-e_3$}}}}}
\put(0,2051){\makebox(0,0)[lb]{\smash{{{\SetFigFont{12}{14.4}{\familydefault}{\mddefault}{\updefault}$e_1-e_2$}}}}}
\end{picture}
}
\end{center}

Then $\sigma$ induces the permutation
 of the set $\{\lambda_1, \lambda_2, \lambda_3\}$, 
 by $\lambda_1 \mapsto
 \lambda_2 \mapsto \lambda_3 \mapsto \lambda_1$, 
 and thus the representations
 $i$, $\operatorname{spin}^+$, 
and $\operatorname{spin}^-$ 
 are mutually equivalent
 by the outer automorphism group of 
 $Spin(8)$ 
 ({\it{triliality}} of $D_4$).  
Hence, 
 without loss of generality,
 we may and do assume 
 that $Spin(8)$ acts on $S^7 \times S^7$
 via $i \oplus \operatorname{spin}^+$.  

First, 
 we consider the action 
 of $Spin(8)$ on the first factor $S^7$
 via the natural representation
 $Spin(8) \overset i \to SO(8)$, 
 which is a transitive action
 and gives rise to a natural diffeomorphism
 $Spin(8)/Spin(7) \simeq S^7$.  

Second,
 we consider the action of the isotropy subgroup $Spin(7)$
 on the second factor $S^7$ via the following composition:
\[
  Spin(7) \hookrightarrow Spin(8) \overset{\operatorname{spin}^+} \to 
 SO(8).  
\]
This action is again transitive,
 and giving a natural diffeomorphism
 $Spin(7)/G_2 \simeq S^7$.  
Thus we have shown that $Spin(8)$ acts 
 transitively 
 on $S^7 \times S^7$ 
 via $i \oplus {\operatorname{spin}^+}$.  
\end{proof}
For the Lie algebra 
 ${\mathfrak {h}}={\mathfrak {o}}(9,1) +{\mathbb{R}}$, 
 the adjoint action 
 of the Lie algebra ${\mathfrak {m}}_H$
 on 
${\mathfrak {n}}^{\sigma} 
={\mathfrak {n}} \cap {\mathfrak {h}}$
 is isomorphic to the natural representation
 of ${\mathfrak {so}}(8)$ on ${\mathbb{R}}^8$.  

We are ready to complete the proof of Proposition \ref{prop:e6}.  
\begin{proof}
[Proof of Proposition \ref{prop:e6}]
Since $\operatorname{rank}_{\mathbb{R}}H=\operatorname{rank}_{\mathbb{R}}G$
 ($=2$), 
 (PP) is equivalent to (QP)
 by Lemma \ref{lem:BPQ}.

The identity component $(M_H)_0$ 
of $M_H$
 is isomorphic to $Spin(8)$, 
 and the adjoint action of $(M_H)_0$
 on ${\mathfrak {n}}^{-\sigma}
 \simeq {\mathfrak {n}}/{\mathfrak {n}}^{\sigma}$
 is isomorphic to the spin representation 
 ${\operatorname{spin}^+}\oplus {\operatorname{spin}^-}$
 of $Spin(8)$ on ${\mathbb{R}}^{16}={\mathbb{R}}^8 \oplus {\mathbb{R}}^8$.  
Thus it induces a transitive action 
 of $Spin(8)$ on $S^7 \times S^7$
 by Lemma \ref{lem:SS}.  
On the other hand,
 since there are two distinct weights of ${\mathfrak {a}}_H$
 on ${\mathfrak {n}}^{-\sigma}$,
 we conclude 
 that the adjoint action of $M_H A_H$ has an open dense orbit 
 in ${\mathfrak {n}}^{-\sigma} \simeq {\mathbb{R}}^{16}$.  
By Theorem \ref{thm:qp}, 
 $(G,H)$ satisfies (QP).  
\end{proof}

\section{Symmetric pair $(G,H)$ with ${\operatorname{rank}}_{\mathbb{R}} H=1$}
\label{sec:rank1}

Since a minimal parabolic subgroup
 of a compact Lie group $K$ is $K$ itself,
 the following proposition is obvious
 by the Iwasawa decomposition
 $G=K A_G N= K P_G$:
\begin{prop}
\label{prop:GK}
Any Riemannian symmetric pair $(G,K)$ satisfies 
 {\rm{(PP)}} and {\rm{(QP)}}.  
\end{prop}
Among reductive symmetric pairs
 $(G,H)$, 
 the Riemannian symmetric pair
 is characterized by the condition 
 ${\operatorname{rank}}_{\mathbb{R}} H=0$.  
In this section, 
 as the \lq\lq{next case}\rq\rq\ 
 of Proposition \ref{prop:GK}, 
 we highlight 
  the case where ${\operatorname{rank}}_{\mathbb{R}}H=1$
 and we give a classification 
 of $({\mathfrak {g}}, {\mathfrak {h}})$
 satisfying (PP), 
 see (E1)--(E4), 
 (G2), or (H1), 
 in Theorem \ref{thm:1.1}.

In Table \ref{tab:5.1}, 
 we give a list of all irreducible symmetric pairs
 $({\mathfrak {g}}, {\mathfrak{h}})$
 with $\operatorname{rank}_{{\mathbb{R}}}{\mathfrak {h}}=1$.  
Since $({\mathfrak {g}}, {\mathfrak{h}})$ and 
 its $c$-dual $({\mathfrak {g}}^c, {\mathfrak{h}})$
 have the same root multiplicities
 $m^{\pm}(\lambda)= \dim {\mathfrak {g}}^{\pm \sigma}({\mathfrak {a}}_H;\lambda)$, 
 we write them in the same row.  
Some of the symmetric pairs
 are labelled as 
 ${\rm{I}}_{{\mathbb{R}}}$, ${\rm{I}}_{{\mathbb{R}}}^c$, 
 $\cdots$, 
for which we give more detailed data in Table \ref{tab:5.2}.  
We are now ready to state
 the main result of this section.  
This completes the proof
 of Theorems \ref{thm:1.1} and \ref{thm:QpPp}
 for the classification of $({\mathfrak {g}}, {\mathfrak{h}})$
 with (PP) and (QP), 
 respectively,
 under the assumption 
 that ${\operatorname{rank}}_{\mathbb{R}} H=1$.  

\begin{prop}
\label{prop:rankH}
Suppose $(G,H)$ is an irreducible symmetric pair
 with $\operatorname{rank}_{{\mathbb{R}}}H=1$.  
\begin{enumerate}
\item[{\rm{1)}}]
The following two conditions {\rm{(i)}} and {\rm{(ii)}}
 on the pair $(G,H)$
 are equivalent:
\begin{enumerate}
\item[{\rm{(i)}}]
The pair $(G,H)$ satisfies {\rm{(QP)}}.  
\item[{\rm{(ii)}}]
The pair $({\mathfrak {g}}, {\mathfrak {h}})$ is one
 of ${\rm{I}}_{{\mathbb{F}}}$, ${\rm{I}}_{{\mathbb{F}}}^c$
 $({\mathbb{F}}={\mathbb{R}}$, ${\mathbb{C}}$, 
 ${\mathbb{H}}$, or ${\mathbb{O}})$, 
${\rm{II}}$, 
${\rm{II}}^c$, 
${\rm{III}}$ or ${\rm{III}}^c$.  
\end{enumerate}
\item[{\rm{2)}}]
The following two conditions {\rm{(iii)}} and {\rm{(iv)}}
 of the pair $(G,H)$
 are equivalent:
\begin{enumerate}
\item[{\rm{(iii)}}]
The pair $(G,H)$ satisfies {\rm{(PP)}}.  
\item[{\rm{(iv)}}]
The pair
 $({\mathfrak {g}}, {\mathfrak {h}})$ is one
 of ${\rm{I}}_{{\mathbb{R}}}^c$, ${\rm{I}}_{{\mathbb{C}}}^c$, 
 ${\rm{I}}_{{\mathbb{H}}}^c$, ${\rm{I}}_{{\mathbb{O}}}^c$,
 ${\rm{II}}^c$, ${\rm{III}}^c$, 
 ${\rm{I}}_{{\mathbb{F}}}$
 $({\mathbb{F}}={\mathbb{R}}, {\mathbb{C}}, {\mathbb{H}})$
 with $p=0$ or $q=1$, 
 ${\rm{II}}$ with $m=1, 2$, 
 or ${\rm{III}}$
 with $m=1,2$.  
\end{enumerate}
\end{enumerate}
\end{prop}
\begin{proof}
We divide the proof
 into three steps.  
\par\noindent
{\bf{Step 1.}}\enspace
For the implication (i) $\Rightarrow$ (ii), 
 we apply Proposition \ref{prop:QPrank}.  
Since $\operatorname{rank}_{{\mathbb{R}}}H=1$, 
 if $(G,H)$ satisfies (QP), 
then $m^-(\lambda)+m^-(2\lambda) \le 1$
 by \eqref{eqn:rn}.  
In light of Table \ref{tab:5.1}, 
 we have shown
 that (i) implies (ii).  
\par\noindent
{\bf{Step 2.}}\enspace
In order to prove the equivalence (iii) $\Leftrightarrow$ (iv), 
 it is sufficient 
 to deal with symmetric pairs $(G,H)$ satisfying (QP)
 because (PP) implies (QP)
 (see Lemma \ref{lem:BPQ}).  
In particular, 
 we may assume 
 that $(G,H)$ satisfies (ii) by Step 1.  
For the pairs $({\mathfrak {g}}, {\mathfrak {h}})$
 satisfying (ii), 
we give a list of the vector spaces ${\mathfrak {n}}^{-\sigma}$
 on which the subalgebras
 $({\mathfrak {m}}_H \cap {\mathfrak {m}}_G)+{\mathfrak {a}}_H
 \subset {\mathfrak {m}}_H + {\mathfrak {a}}_H$
 act via the adjoint representation 
 in Table \ref{tab:5.2}.  
For ${\rm{I}}_{{\mathbb{F}}}$
 $({\mathbb{F}}={\mathbb{R}}$, ${\mathbb{C}}$, 
 or ${\mathbb{H}})$
 in this table, 
 the action of ${\mathfrak {u}}(|p-q|;{\mathbb{F}})$
 on ${\mathbb{F}}^q$
 is trivial 
if $p \ge q$ 
 and is the natural action
 on the first factor of the decomposition 
 ${\mathbb{F}}^q={\mathbb{F}}^{q-p} \oplus {\mathbb{F}}^p$
 if $q \ge p$.  
In view of this table,
 we see 
 that $(M_H \cap M_G)A_H$ has an open orbit
 in ${\mathfrak {n}}^{-\sigma}$
 if and only if (iv) holds.  
Hence the equivalence (iii) $\Leftrightarrow$ (iv)
 follows from Theorem \ref{thm:pp}.  
\par\noindent
{\bf{Step 3.}}\enspace
The converse implication (i) $\Leftarrow$ (ii) follows from Step 2 and Proposition \ref{prop:cdual}
 because ${\rm{I}}_{{\mathbb{F}}}^c$, 
 ${\rm{II}}^c$, and ${\rm{III}}^c$ are the $c$-duals
 of ${\rm{I}}_{{\mathbb{F}}}$, 
 ${\rm{II}}$, and ${\rm{III}}$, 
respectively.  
\end{proof}
\begin{table}[H]
\caption{Irreducible symmetric pairs
 $({\mathfrak{g}}, {\mathfrak {h}})$
 with $\operatorname{rank}_{\mathbb{R}}{\mathfrak {h}}=1$}
\label{tab:5.1}
\centering
\begin{tabular}{|c|c|c|c|}
\hline
&
$\begin{matrix}
{\mathfrak {g}}
\\
{\mathfrak {g}}^c
\end{matrix}$ 
&  ${\mathfrak {h}}$ 
&  $\begin{pmatrix} m^+(\lambda) & m^+(2\lambda) \\ m^-(\lambda) & m^-(2\lambda)\end{pmatrix}$
\\
\hline
$\begin{matrix}
{\rm{I}}_{\mathbb{R}}
\\ 
{\rm{I}}_{\mathbb{R}}^c \end{matrix}$ 
&$\begin{matrix} {\mathfrak {so}}(p+1,q+1) 
\\ 
{\mathfrak {so}}(p+q+1,1)\end{matrix}$  & ${\mathfrak {so}}(q) + {\mathfrak {so}}(p+1,1)$ & $\begin{pmatrix} p & 0 \\ q & 0 \end{pmatrix}$ \\
\hline
$\begin{matrix} 
{\rm{I}}_{\mathbb{C}} 
\\
{\rm{I}}_{\mathbb{C}}^c
\end{matrix}$  
&
$\begin{matrix} 
{\mathfrak {u}}(p+1,q+1) 
\\
{\mathfrak {u}}(p+q+1,1)\end{matrix}$  
& ${\mathfrak {u}}(q) + {\mathfrak {u}}(p+1,1)$ 
& $\begin{pmatrix} 2p & 1 \\ 2q & 0 \end{pmatrix}$
\\
\hline
$\begin{matrix}
{\rm{I}}_{\mathbb{H}}
\\
{\rm{I}}_{\mathbb{H}}^c
\end{matrix}$ 
&$\begin{matrix}
{\mathfrak {sp}}(p+1,q+1) 
\\
{\mathfrak {sp}}(p+q+1,1)\end{matrix}$  & ${\mathfrak {sp}}(q) + {\mathfrak {sp}}(p+1,1)$ 
& $\begin{pmatrix} 4p & 3 \\ 4q & 0 \end{pmatrix}$\\
\hline
${\rm{I}}_{\mathbb{O}}={\rm{I}}_{\mathbb{O}}^c$
& ${\mathfrak {f}}_{4(-20)}$ & ${\mathfrak {so}}(8,1)$ & $\begin{pmatrix} 0 & 7 \\ 8 & 0 \end{pmatrix}$\\
\hline
&$\begin{matrix} {\mathfrak {sl}}(m+2,{\mathbb{R}}) \\ {\mathfrak {sl}}(m+1,{\mathbb{R}})\end{matrix}$  & ${\mathfrak {so}}(m+1,1)$ 
& $\begin{pmatrix} m & 0 \\ m & 1 \end{pmatrix}$\\
\hline
&$\begin{matrix} {\mathfrak {sp}}(m+2,{\mathbb{R}}) \\ {\mathfrak {sp}}(m+1,{\mathbb{R}})\end{matrix}$
& ${\mathfrak {u}}(m+1,1)$ 
& $\begin{pmatrix} 2m & 1 \\ 2m & 2 \end{pmatrix}$\\
\hline
&$\begin{matrix} {\mathfrak {f}}_{4(4)} \\ {\mathfrak {f}}_{4(-20)}\end{matrix}$  & ${\mathfrak {sp}}(2,1)+{\mathfrak {su}}(2)$
& $\begin{pmatrix} 4 & 3 \\ 4 & 4 \end{pmatrix}$\\
\hline
$\begin{matrix} {\rm{II}} 
\\
{\rm{II}}^c \end{matrix}$
&$\begin{matrix} {\mathfrak {so}}(m+2,{\mathbb{C}}) 
\\
{\mathfrak {h}}+{\mathfrak{h}}\end{matrix}$
& ${\mathfrak {so}}(m+1,1)$
& $\begin{pmatrix} m & 0 \\ m & 0 \end{pmatrix}$\\
\hline
&$\begin{matrix} {\mathfrak {sl}}(m+2,{\mathbb{C}}) 
\\ {\mathfrak {h}}+{\mathfrak{h}}\end{matrix}$
  & ${\mathfrak {su}}(m+1,1)$
& $\begin{pmatrix} 2m & 1 \\ 2m & 1 \end{pmatrix}$\\
\hline
&$\begin{matrix} 
{\mathfrak {sp}}(m+2,{\mathbb{C}}) 
\\ 
{\mathfrak {h}}+{\mathfrak{h}}\end{matrix}$  & ${\mathfrak {sp}}(m+1,1)$ 
& $\begin{pmatrix} 4m & 3 \\ 4m & 3 \end{pmatrix}$\\
\hline
&$\begin{matrix} {\mathfrak {f}}_{4}({\mathbb{C}}) 
\\ {\mathfrak {h}}+{\mathfrak {h}}\end{matrix}$
  & ${\mathfrak {f}}_{4(-20)}$ 
& $\begin{pmatrix}8 & 7 \\ 8 & 7 \end{pmatrix}$\\
\hline
$\begin{matrix} 
{\rm{III}} 
\\ 
{\rm{III}}^c\end{matrix}$
&$\begin{matrix} {\mathfrak {so}}^{\ast}(2m+4) 
\\ {\mathfrak {so}}(2m+2,2)\end{matrix}$
  & ${\mathfrak {u}}(m+1,1)$
 & $\begin{pmatrix}2m & 1 \\ 2m & 0 \end{pmatrix}$\\
\hline
&
$\begin{matrix} {\mathfrak {su}}^{\ast}(2m+4) 
\\ {\mathfrak {su}}(2m+2,2)\end{matrix}$
  & ${\mathfrak {sp}}(m+1,1)$
 & $\begin{pmatrix}4m & 3 \\ 4m & 1 \end{pmatrix}$\\
\hline
&$\begin{matrix} {\mathfrak {e}}_{6(-26)} 
\\ {\mathfrak {e}}_{6(-14)}\end{matrix}$
  & ${\mathfrak {f}}_{4(-20)}$
 & $\begin{pmatrix}8 & 7 \\ 8 & 1 \end{pmatrix}$\\
\hline
&${\mathfrak {sl}}(3,{\mathbb{C}})$
  & ${\mathfrak {so}}(3,{\mathbb{C}})$
 & $\begin{pmatrix} 2 & 0 \\ 2 & 2 \end{pmatrix}$\\
\hline
&$\begin{matrix} {\mathfrak {su}}(3,3) 
\\ {\mathfrak {su}}^{\ast}(6) \end{matrix}$
  & ${\mathfrak {so}}^{\ast}(6)$
 & $\begin{pmatrix} 4 & 1 \\ 4 & 3 \end{pmatrix}$\\
\hline
&$\begin{matrix} {\mathfrak {e}}_{6(2)} 
\\ {\mathfrak {e}}_{6(-26)}\end{matrix}$
  & ${\mathfrak {sp}}(3,1)$
 & $\begin{pmatrix}8 & 3 \\ 8 & 5 \end{pmatrix}$\\
\hline
\end{tabular}\\
\end{table}

\begin{table}[H]
\caption{Irreducible symmetric pairs with (QP)
 and 
$\operatorname{rank}_{{\mathbb{R}}}{\mathfrak {h}}=1$}
\label{tab:5.2}
\centering
\begin{tabular}{|c|c|c|c|}
\hline
& (${\mathfrak {m}}_H \cap {\mathfrak {m}}_G)+{\mathfrak {a}}_H$  
&  ${\mathfrak {m}}_H + {\mathfrak {a}}_H$
& ${\mathfrak {n}}^{-\sigma}$
\\
\hline
  ${\rm{I}}_{\mathbb{R}}$
&   ${\mathfrak {o}}(|p-q|)+{\mathbb{R}}$ 
&         &       \\
\cline{1-1}\cline{2-2}
${\rm{I}}_{\mathbb{R}}^c$ &  ${\mathfrak {o}}(q) + {\mathfrak {o}}(p)+{\mathbb{R}}$ 
& \raisebox{1.5ex}[0cm][0cm]{${\mathfrak {o}}(q)+{\mathfrak {o}}(p)+{\mathbb{R}}$}
 &\raisebox{1.5ex}[0cm][0cm]{  ${\mathbb{R}}^q$}  \\
\hline
${\rm{I}}_{\mathbb{C}}$  
&   ${\mathfrak {u}}(|p-q|) +(\sqrt{-1}{\mathbb{R}})^{\min(p,q)}+{\mathbb{R}}$
  &         &       \\
\cline{1-1}\cline{2-2} 
 ${\rm{I}}_{\mathbb{C}}^c$ &  ${\mathfrak {u}}(q)+{\mathfrak {u}}(p)+{\mathbb{C}}$
 & \raisebox{1.5ex}[0cm][0cm]{${\mathfrak {u}}(q)+{\mathfrak {u}}(p) +{\mathbb{C}}$}  
 & \raisebox{1.5ex}[0cm][0cm]{  ${\mathbb{C}}^q$}  \\
\hline
${\rm{I}}_{\mathbb{H}}$
 &  ${\mathfrak {sp}}(|p-q|)+{\mathfrak {sp}}(1)^{\min(p,q)}+{\mathbb{R}}$
 &         &       \\
\cline{1-1}\cline{2-2}
 ${\rm{I}}_{\mathbb{H}}^c$& ${\mathfrak {sp}}(q)+{\mathfrak {sp}}(p)+
{\mathbb{H}}$
 & \raisebox{1.5ex}[0cm][0cm]{${\mathfrak {sp}}(q)+ {\mathfrak {sp}}(p)+{\mathbb{H}}$}
 & \raisebox{1.5ex}[0cm][0cm]{  ${\mathbb{H}}^q$}  \\
\hline
${\rm{I}}_{\mathbb{O}}={\rm{I}}_{\mathbb{O}}^c$
&  ${\mathfrak{spin}}(7) + {\mathbb{R}}$
&  ${\mathfrak{spin}}(7) + {\mathbb{R}}$
 &    ${\mathbb{R}}^8$ \\
\hline
 ${\rm{II}}$  &   ${\mathbb{T}}^{[\frac m 2]}+{\mathbb{R}}$
 &         &       \\
\cline{1-1}\cline{2-2}
  ${\rm{II}}^c$& ${\mathfrak {o}}(m)+{\mathbb{R}}$ 
 & \raisebox{1.5ex}[0cm][0cm]{  ${\mathfrak {o}}(m)+{\mathbb{R}}$} 
 &\raisebox{1.5ex}[0cm][0cm] {${\mathbb{R}}^m$ } \\
\hline
 ${\rm{III}}$ &  ${\mathfrak {sp}}(1)^{[\frac m 2]}+{\mathbb{C}}$
 &         &       \\
\cline{1-1}\cline{2-2}
 ${\rm{III}}^c$& ${\mathfrak {u}}(m)+{\mathbb{R}}$ & \raisebox{1.5ex}[0cm][0cm]{${\mathfrak {u}}(m)+{\mathbb{C}}$} 
&   \raisebox{1.5ex}[0cm][0cm]{ ${\mathbb{C}}^m$} \\
\hline
\end{tabular}\\
\end{table}

\section{Associated symmetric pairs of non-$K_{\varepsilon}$-family}
\label{sec:nonKe}
In this section and the next section,
 we complete the proof of the classification 
 of symmetric pairs $({\mathfrak {g}}, {\mathfrak {h}})$
 satisfying (PP) (or (QP))
 and ${\operatorname{rank}}_{\mathbb{R}}H \ge 2$.  
For this, 
 we make use of the $K_{\varepsilon}$-family 
 introduced in \cite{OS1}
 (See Definition \ref{def:Ke} below), 
 which is 
 a fairly large class
 of reductive symmetric pairs.

We recall
 that if a reductive symmetric pair
 $({\mathfrak {g}}, {\mathfrak {h}})$
 is defined
 by an involutive automorphism $\sigma$
 of ${\mathfrak {g}}$
 then we can define another involution $\sigma \theta$
 by taking a Cartan involution 
 $\theta$ 
 commuting with $\sigma$.  
The symmetric pair $({\mathfrak {g}}, {\mathfrak {g}}^{\sigma \theta})$
 is called the {\it{associated symmetric pair}}
 of $({\mathfrak {g}}, {\mathfrak {h}})
 \equiv ({\mathfrak {g}}, {\mathfrak {g}}^{\sigma})$.  
Our strategy 
 is based on the following ideas.  
\begin{enumerate}
\item[(1)]
Very few pairs 
 $({\mathfrak {g}}, {\mathfrak {h}})
 \equiv({\mathfrak {g}}, {\mathfrak {g}}^{\sigma})$
 satisfy (QP)
 if $({\mathfrak {g}}, {\mathfrak {g}}^{\sigma\theta})$
 does not belong to the 
 $K_{\varepsilon}$-family 
 (Proposition \ref{prop:nonKe}).  
\item[(2)]
$\operatorname{rank}_{\mathbb{R}}G =\operatorname{rank}_{\mathbb{R}}H$
  if $({\mathfrak {g}}, {\mathfrak {g}}^{\sigma\theta})$
 belongs to the $K_{\varepsilon}$-family.  
\end{enumerate}
In this section we treat the case
 where the associated symmetric pair $({\mathfrak {g}}, {\mathfrak {g}}^{\sigma\theta})$
 does not belong to the $K_{\varepsilon}$-family, 
 and in the next section
 we discuss
 the opposite case
 where $({\mathfrak {g}}, {\mathfrak {g}}^{\sigma\theta})$
 belongs to the $K_{\varepsilon}$-family.  
To be more precise,
 let us review the definition of $K_{\varepsilon}$-family.  
Suppose ${\mathfrak {a}}_G$ 
 is a maximal abelian subspace
 of ${\mathfrak {g}}^{-\theta}$
 as before.  
\begin{df}
\label{def:Ke}
{\rm{
A map $\varepsilon :\Sigma ({\mathfrak {g}}, {\mathfrak {a}}_G)
 \cup \{0\} \to \{\pm 1\}$
 is said to be a {\it{signature}}
 if
\[
   \varepsilon(\alpha+ \beta)
   =
   \varepsilon(\alpha)\varepsilon(\beta)
\quad
 \text{for any}
\,\,
 \alpha, \beta 
\,\,
\text{and }
\,\, \alpha+ \beta
\in \Sigma ({\mathfrak {g}}, {\mathfrak {a}}_G)
 \cup \{0\}.  
\]
We note that 
 $\varepsilon(0)=1$
 and $\varepsilon(\alpha)=\varepsilon(-\alpha)$
 for any $\alpha \in \Sigma ({\mathfrak {g}}, {\mathfrak {a}}_G)$.  
We define another involution
 $\theta_{\varepsilon}$
 by 
\[
  \theta_{\varepsilon} (X)
  :=
  \varepsilon(\alpha) \theta(X)
  \quad \text{for }\,\, X \in {\mathfrak {g}}({\mathfrak {a}}_G;\alpha), 
\]
 and set ${\mathfrak {k}}_{\varepsilon} 
:=\{X \in {\mathfrak {g}}:\theta_{\varepsilon}(X)=X\}
$.  
If $\varepsilon \equiv 1$
 then ${\mathfrak {k}}_{\varepsilon}={\mathfrak {k}}$.  
The reductive symmetric pairs
 $\{({\mathfrak {g}}, {\mathfrak {k}}_{\varepsilon}):
\varepsilon\text{ is a signature}\}$
 are called the {\it{$K_{\varepsilon}$-family}}.  
}}
\end{df}
Here is the main result of this section:
\begin{prop}
\label{prop:nonKe}
Let $(G,H)$ be an irreducible symmetric pair
 defined by an involution $\sigma$.  
Assume that the following two conditions are fulfilled:
\begin{align}
&\text{The associated pair
 $({\mathfrak {g}}, {\mathfrak {g}}^{\sigma\theta})$
 does not belong to the $K_{\varepsilon}$-family.  }
\label{eqn:nonKe(a)}
\\
&\operatorname{rank}_{\mathbb{R}}H >1.  
\label{eqn:nonKe(c)}
\end{align}
Then either $\operatorname{rank}_{\mathbb{R}}H
 < \# \Delta({\mathfrak {n}}^{-\sigma})$
 or 
\begin{equation}
\label{eqn:nonKe(b)}
\text{
 $({\mathfrak {g}}, {\mathfrak {h}})$
 is a symmetric pair 
treated in Propositions 
 \ref{prop:upq}, \ref{prop:somn} and \ref{prop:sostar}.  
}
\end{equation}
In particular,
 there is no irreducible symmetric pair
  $({\mathfrak {g}}, {\mathfrak {h}})$
 with \eqref{eqn:nonKe(a)}
 and \eqref{eqn:nonKe(c)}
 other than 
 those listed in Propositions 
 \ref{prop:upq}, \ref{prop:somn} and \ref{prop:sostar}.  
\end{prop}
\begin{proof}
Suppose $({\mathfrak {g}}, {\mathfrak {g}}^{\sigma\theta})$
 does not belong
 to the $K_{\varepsilon}$-family.  
Then by using the classification \cite[Table V]{OS}
and by computing the correspondence
$({\mathfrak {g}}, {\mathfrak {h}})
\equiv({\mathfrak {g}}, {\mathfrak {g}}^{\sigma})
\leftrightarrow
({\mathfrak {g}}, {\mathfrak {g}}^{\sigma\theta})$, 
 we observe
 that either $H$ is a simple Lie group
 up to a compact torus
 or \eqref{eqn:nonKe(b)} holds.

{}From now on,
 we assume
 that the irreducible symmetric pair
 $({\mathfrak {g}}, {\mathfrak {h}})$
 satisfies \eqref{eqn:nonKe(a)} and \eqref{eqn:nonKe(c)}
 but does not satisfy \eqref{eqn:nonKe(b)}.  
Then the restricted root system
 $\Sigma({\mathfrak {h}}, {\mathfrak {a}}_H)$
 is irreducible.  
Then, 
by Lemma \ref{lem:2.6},
 the condition $\operatorname{rank}_{\mathbb{R}}H
 \ge \# \Delta({\mathfrak {n}}^{-\sigma})$
 gives strong constraints
 on both the irreducible root system 
 $\Sigma({\mathfrak {h}}, {\mathfrak {a}}_H)$
 and the set $\Delta({\mathfrak {n}}^{-\sigma})$, 
 namely, 
 $\operatorname{rank}_{\mathbb{R}}H
 \ge \# \Delta({\mathfrak {n}}^{-\sigma})$
 implies that $\Sigma({\mathfrak {h}}, {\mathfrak {a}}_H)$
 is one of type $B_l$, $C_l$, $D_l$ or $BC_l$
 and that $\Delta({\mathfrak {n}}^{-\sigma})$
 is contained 
 in either $\{\pm e_i: 1 \le i \le l\}$
 or $\{\pm 2 e_i: 1 \le i \le l\}$.  
Furthermore, 
 $m^-(\lambda) \le 1$
 and $m^-(\lambda)m^-(2\lambda)=0$
 for all $\lambda$.  

In turn, 
 in view of the classification
 of irreducible symmetric pairs
 satisfying \eqref{eqn:nonKe(a)}
 and the formulae
 for the multiplicities
 $m^-(\lambda_i)$ and $m^-(2\lambda_i)$
 in \cite[Table V]{OS}, 
 we see 
 that this does not happen.  
To verify it,
 we remark that the role of $({\mathfrak {g}}, {\mathfrak {g}}^{\sigma})$
 and $({\mathfrak {g}}, {\mathfrak {g}}^{\sigma\theta})$
 in their table
 is opposite to our notation here,
 but the role 
 of the multiplicities $m^{\pm}(\lambda)$
 is the same.  
With this remark in mind,
 we obtain the following small list 
{}from \cite[Table V]{OS}
 by picking up those having the above constraints
 on $\Sigma({\mathfrak {h}}, {\mathfrak {a}}_H)$
 and $\Delta({\mathfrak {n}}^{-\sigma})$
 and by skipping those belonging
 to the families
 in Proposition \ref{prop:upq}, \ref{prop:somn} and \ref{prop:sostar}:  
\begin{alignat*}{3}
&{\mathfrak{g}}
&&{\mathfrak{g}}^{\sigma\theta}
&&{\mathfrak{g}}^{\sigma}={\mathfrak{h}}
\\
&{\mathfrak{sl}}(4,{\mathbb{R}})
\qquad
&&{\mathfrak{sl}}(2,{\mathbb{C}})+\sqrt{-1}{\mathbb{R}}
\qquad
&&{\mathfrak{sp}}(2,{\mathbb{R}})
\\
&{\mathfrak{su}}(2,2)
&&{\mathfrak{so}}^{\ast}(4)
&&{\mathfrak{sp}}(2,{\mathbb{R}})
\\
&{\mathfrak{so}}^{\ast}(8)
&&{\mathfrak{so}}^{\ast}(4)+{\mathfrak{so}}^{\ast}(4)
&&{\mathfrak{u}}(2,2)
\\
&{\mathfrak{so}}(4,4)
&&{\mathfrak{u}}(2,2)
&&{\mathfrak{u}}(2,2)
\\
&{\mathfrak{sl}}(4,{\mathbb{C}})
&&{\mathfrak{su}}^{\ast}(4)
&&{\mathfrak{sp}}(2,{\mathbb{C}})
\end{alignat*}
However, 
 these exceptional cases
 are actually included
 in the family of symmetric pairs
 in Propositions \ref{prop:upq}
 and \ref{prop:somn}
 via the following isomorphisms:
\begin{align*}
({\mathfrak {sl}}(4,{\mathbb{R}}), {\mathfrak {sp}}(2,{\mathbb{R}}))
\simeq 
&({\mathfrak {so}}(3,3), {\mathfrak {so}}(3,2)), 
\\
({\mathfrak {su}}(2,2), {\mathfrak {sp}}(2,{\mathbb{R}}))
\simeq
& ({\mathfrak {so}}(4,2), {\mathfrak {so}}(3,2)), 
\\
({\mathfrak {so}}^{\ast}(8), {\mathfrak {u}}(2,2))
\simeq
& ({\mathfrak {so}}(6,2), {\mathfrak {so}}(4,2)+{\mathfrak {so}}(2)), 
\\
({\mathfrak {so}}(4,4), {\mathfrak {u}}(2,2))
\simeq
& ({\mathfrak {so}}(4,4), {\mathfrak {so}}(4,2)+{\mathfrak {so}}(2)), 
\\
({\mathfrak {sl}}(4,{\mathbb{C}}), {\mathfrak {sp}}(2,{\mathbb{C}}))
\simeq
& ({\mathfrak {so}}(6,{\mathbb{C}}), {\mathfrak {so}}(5,{\mathbb{C}})).   
\end{align*}
Thus we have proved
 that $\operatorname{rank}_{\mathbb{R}}H < \#
\Delta ({\mathfrak {n}}^{-\sigma})$
 if \eqref{eqn:nonKe(a)}, and \eqref{eqn:nonKe(c)} 
 are satisfied
 and if \eqref{eqn:nonKe(b)} is not satisfied.  
\end{proof}

\section{Associated symmetric pairs of $K_{\varepsilon}$-family}
\label{sec:Ke}
In this section we consider irreducible symmetric pairs
$({\mathfrak {g}}, {\mathfrak {h}})
\equiv({\mathfrak {g}}, {\mathfrak {g}}^{\sigma})$
 such that the associated symmetric pair $({\mathfrak {g}}, {\mathfrak {g}}^{\sigma\theta})$
 belongs to the $K_{\varepsilon}$-family.  
In this case,
 $\operatorname{rank}_{\mathbb{R}}H=\operatorname{rank}_{\mathbb{R}} G$
 holds from the definition 
 of $K_{\varepsilon}$-family,
 and consequently,
 the condition {\rm{(QP)}} is equivalent to {\rm{(PP)}}
 by Lemma \ref{lem:BPQ}.  

Let ${\mathfrak{g}}_{\mathbb{C}}$
 be the complexification of ${\mathfrak{g}}$.  
For a simple Lie algebra ${\mathfrak{g}}$
 over ${\mathbb{R}}$, 
 ${\mathfrak{g}}_{\mathbb{C}}$ is a complex simple Lie algebra
 if and only if ${\mathfrak{g}}$ itself does not
 carry a complex Lie algebra structure.  
We divide the proof
 into the following three cases:
\par\noindent
{\bf{Case 1.}}\enspace
${\mathfrak{g}}_{\mathbb{C}}$ is not simple.  
\par\noindent
{\bf{Case 2.}}\enspace
${\mathfrak{g}}_{\mathbb{C}}$ is a simple
 classical Lie algebra.  
\par\noindent
{\bf{Case 3.}}\enspace
${\mathfrak{g}}_{\mathbb{C}}$ is a simple
 exceptional Lie algebra.  

In Case 1, 
 the pair $({\mathfrak{g}}, {\mathfrak{h}})$
 was treated 
in Proposition \ref{prop:cpx}.  
In fact, 
 ${\mathfrak{g}}$ is a complex simple Lie algebra.  
Further, 
 ${\mathfrak{g}}^{\sigma \theta}$
 is a real form of ${\mathfrak{g}}$
 as noted in \cite[Appendix]{OS1}, 
 and consequently 
 ${\mathfrak {h}}={\mathfrak {g}}^{\sigma}$
 is a complex Lie subalgebra.  
Hence $({\mathfrak{g}}, {\mathfrak{h}})$
 is a complex symmetric pair
 such that $\operatorname{rank}{\mathfrak{h}}
=\operatorname{rank}{\mathfrak{g}}$.  
\par\noindent
{\bf{Case 2.}}\enspace
Suppose that ${\mathfrak{g}}_{\mathbb{C}}$ is a classical simple Lie algebra.

By the classification of $K_{\varepsilon}$-family
 (see \cite[Table 1]{OS}), 
 the pair $({\mathfrak{g}}, {\mathfrak{h}})$ is
 one of the following pairs
 up to the center of ${\mathfrak{g}}$.  
\begin{align*}
 &({\mathfrak{gl}}(p+q, {\mathbb{F}}), 
   {\mathfrak{gl}}(p, {\mathbb{F}})+{\mathfrak{gl}}(q, {\mathbb{F}})),
 \quad
 {\mathbb{F}}={\mathbb{R}}, {\mathbb{H}}, 
\\
&({\mathfrak{sp}}(p+q, {\mathbb{R}}), 
{\mathfrak{sp}}(p, {\mathbb{R}})
+{\mathfrak{sp}}(q, {\mathbb{R}})), 
\\
&({\mathfrak{u}}(n,n;{\mathbb{F}}), 
  {\mathfrak{gl}}(n, {\mathbb{F}})), 
\quad
{\mathbb{F}}={\mathbb{R}}, {\mathbb{C}}, {\mathbb{H}}, 
\\
&({\mathfrak{sp}}(n, {\mathbb{R}}),
 {\mathfrak{gl}}(n, {\mathbb{R}})),
\\
& ({\mathfrak{so}}^{\ast}(4n), {\mathfrak{gl}}(n, {\mathbb{H}})),
\\
\intertext{or the following two families}
&
({\mathfrak{u}}(i+j, k+l;{\mathbb{F}}), 
 {\mathfrak{u}}(i, k;{\mathbb{F}})+{\mathfrak{u}}(j, l;{\mathbb{F}})),
\quad
{\mathbb{F}}={\mathbb{R}}, {\mathbb{C}}, {\mathbb{H}}, 
\\
&
({\mathfrak{o}}^{\ast}(2p+2q), {\mathfrak{o}}^{\ast}(2p)+
{\mathfrak{o}}^{\ast}(2q)).  
\end{align*}
In the last two cases,
 the condition
 that $({\mathfrak {g}}, {\mathfrak {g}}^{\sigma \theta})$
 belongs to the $K_{\varepsilon}$-family imposes
 certain constraints
 on the parameters
 ({\it{e.g.}} $pq$ is even in the last case).

The first five cases were treated
 in Propositions
 \ref{prop:glgl},
 \ref{prop:sp},  
 \ref{prop:ugl}, 
 and \ref{prop:ug2}.  
The last two cases are covered
 by Propositions \ref{prop:upq} and \ref{prop:sostar}
 without constraints 
 on parameters, 
 respectively.  
Thus there is no \lq\lq{new}\rq\rq\ symmetric pair
 $({\mathfrak {g}}, {\mathfrak {h}})$
 that satisfies (QP).  
\vskip 1pc
\par\noindent
{\bf{Case 3.}}\enspace
${\mathfrak{g}}_{\mathbb{C}}$ is an exceptional simple Lie algebra.

In this case,
 we prove the following:
\begin{prop}
\label{prop:except}
Let $({\mathfrak {g}},{\mathfrak {h}})\equiv
({\mathfrak {g}}, {\mathfrak {g}}^{\sigma})$
 be a symmetric pair
 such that its associated symmetric pair
 $({\mathfrak {g}}, {\mathfrak {g}}^{\sigma\theta})$
 belongs to the $K_{\varepsilon}$-family.  
If ${\mathfrak {g}}_{\mathbb{C}}$ is a simple exceptional Lie algebra,
 then the following three conditions are equivalent:
\begin{enumerate}
\item[{\rm{(i)}}]
$({\mathfrak {g}},{\mathfrak {h}})$ satisfies {\rm{(QP)}}.  
\item[{\rm{(ii)}}]
$({\mathfrak {g}},{\mathfrak {h}})$ satisfies {\rm{(PP)}}.  
\item[{\rm{(iii)}}]
$({\mathfrak {g}},{\mathfrak {h}})$
 is either 
$({\mathfrak {e}}_{6(-26)},{\mathfrak {so}}(9,1)+{\mathbb{R}})$
or $({\mathfrak {f}}_{4(-20)},{\mathfrak {so}}(8,1))$.  
\end{enumerate}
\end{prop}
\begin{proof}
The equivalence (i) $\Leftrightarrow$ (ii) holds
 because ${\operatorname{rank}}_{\mathbb{R}}G=
 {\operatorname{rank}}_{\mathbb{R}}H$.  
We have already proved the implication 
 (iii) $\Rightarrow$ (ii)
 in Propositions \ref{prop:e6}
 and \ref{prop:rankH}.  
The remaining implication 
 (i) $\Rightarrow$ (iii) is deduced from 
 the following two lemmas.  
\end{proof}
\begin{lemma}
\label{lem:eso82}
The symmetric pair
 $({\mathfrak {e}}_{6(-14)}, {\mathfrak {so}}(8,2)+\sqrt{-1}{\mathbb{R}})$
 does not satisfy {\rm{(QP)}}.  
\end{lemma}
\begin{proof}
We take the standard basis $\{e_1, e_2\}$
 of ${\mathfrak {a}}_H^{\ast}={\mathfrak {a}}_G^{\ast}$
 in such a way 
 that $\Sigma^+({\mathfrak {h}}, {\mathfrak {a}}_H)
=\{e_1, e_2, e_1 \pm e_2\}$.  
Then the root multiplicities $m^{\pm}(\lambda)$ are given as follows:
\begin{alignat*}{4}
&\lambda
\quad
&&\pm e_i\,(i=1,2)
\quad
&&
\pm 2e_i\,(i=1,2)
\quad
&&
\pm e_1 \pm e_2
\\
&m^+(\lambda)\qquad
&&
\quad 6
&&
\quad 0
&&
\quad 1
\\
&m^-(\lambda)
&&
\quad 0
&&
\quad 1
&&
\quad 7
\end{alignat*}
Thus
 $\Delta({\mathfrak {n}}^{-\sigma})
=\{2e_1, 2e_2, e_1 \pm e_2\}$, 
 and $\# \Delta({\mathfrak {n}}^{-\sigma})=4
> \operatorname{rank}_{\mathbb{R}}H=2$.  
Now the lemma follows from Proposition \ref{prop:QPrank}.  
\end{proof}

For the remaining cases,
 we use Proposition \ref{prop:QPineq}
 as an easy-to-check sufficient condition for {\rm{(QP)}}.  
We obtain the following:
\begin{lemma}
\label{lem:Kex}
Let ${\mathfrak {g}}$ be an exceptional simple Lie algebra
 and $({\mathfrak {g}}, {\mathfrak {h}})$
 a symmetric pair
such that
 its associated symmetric pair
 belongs to the $K_{\varepsilon}$-family.  
Then the inequality \eqref{eqn:QPineq} holds
 if and only if the pair $({\mathfrak {g}}, {\mathfrak {h}})$
 is one of the following:
\[
({\mathfrak {e}}_{6(-14)}, {\mathfrak {so}}(8,2) +\sqrt{-1}{\mathbb{R}}), 
\quad
({\mathfrak {e}}_{6(-26)}, {\mathfrak {so}}(9,1) +{\mathbb{R}}), 
\quad
({\mathfrak {f}}_{4(-20)}, {\mathfrak {so}}(8,1)).  
\]
\end{lemma}
\begin{proof}
In Table \ref{tab:7.1}, 
 we list all the symmetric pairs
 $({\mathfrak {g}}, {\mathfrak {h}})
\equiv({\mathfrak {g}}, {\mathfrak {g}}^{\sigma})$
 such that ${\mathfrak {g}}_{\mathbb{C}}$
 is an exceptional simple Lie algebra
 and that $({\mathfrak {g}}, {\mathfrak {g}}^{\sigma\theta})$
 belongs to the $K_{\varepsilon}$-family.  
In this table, 
we also list the data $m(G)$
 (see \eqref{eqn:mG});
 $n(G)$, $n(H)$
 (see \eqref{eqn:nG}), 
 and 
 $\operatorname{rank}_{\mathbb{R}}G
(=
  \operatorname{rank}_{\mathbb{R}}H$).  
Now Lemma \ref{lem:Kex} follows from 
 the computation of the signature of $n(G)-n(H)-m(G) {\operatorname{rank}}_{\mathbb{R}}H$.  
\end{proof}

\begin{table}[H]
\caption{Exceptional symmetric pairs
 $({\mathfrak {g}}, {\mathfrak {h}})
\equiv({\mathfrak {g}}, {\mathfrak {g}}^{\sigma})$
 with $({\mathfrak {g}}, {\mathfrak {g}}^{\sigma\theta})$
in $K_{\varepsilon}$-family}
\label{tab:7.1}
\begin{tabular}{|c|c|c|c|c|c|c|}
\hline
$G$ 
& $\operatorname{rank}_{\mathbb{R}} G$ 
& $m(G)$ 
& $n(G)$ 
& $H$ 
& $n(H)$
& $m(G) \operatorname{rank}_{\mathbb{R}} G$ v.s. $n(G)-n(H)$ 
\\
\hline
&        
&      
&      
& ${\mathfrak {sl}}(6,{\mathbb{R}})+{\mathfrak {sl}}(2,{\mathbb{R}})$
& 16  
& $6<20$     
\\
\cline{5-5}\cline{6-6}\cline{7-7}
\raisebox{1.5ex}[0cm][0cm]{${\mathfrak{e}}_{6(6)}$} 
&\raisebox{1.5ex}[0cm][0cm]{6}   
&\raisebox{1.5ex}[0cm][0cm]{1}   
&\raisebox{1.5ex}[0cm][0cm]{36}  
&{${\mathfrak{so}}(5,5)+{\mathbb{R}}$}
&{20}  
&{$6<16$}               
\\
\hline
&        
&      
&      
& ${\mathfrak{so}}(6,4)+\sqrt{-1}{\mathbb {R}}$ 
&  20  
&    $8<16$       
\\
\cline{5-5}\cline{6-6}\cline{7-7}
\raisebox{1.5ex}[0cm][0cm]{${\mathfrak{e}}_{6(2)}$} 
&    \raisebox{1.5ex}[0cm][0cm]{4}   
&  \raisebox{1.5ex}[0cm][0cm]{2}   
&  \raisebox{1.5ex}[0cm][0cm]{36}  
& ${\mathfrak{su}}(3,3)+{\mathfrak {sl}}(2,{\mathbb{R}})$ 
&  16  
&    $8<20$                
\\
\hline
&        
&      
&      
& ${\mathfrak{su}}(5,1)+{\mathfrak {sl}}(2,{\mathbb{R}})$ 
&  10  
&    $16<20$                 
\\
\cline{5-5}\cline{6-6}\cline{7-7}
\raisebox{1.5ex}[0cm][0cm]{${\mathfrak{e}}_{6(-14)}$} 
&  \raisebox{1.5ex}[0cm][0cm]{2}   
&  \raisebox{1.5ex}[0cm][0cm]{8}   
&  \raisebox{1.5ex}[0cm][0cm]{30}      
& ${\mathfrak{so}}(8,2)+\sqrt{-1}{\mathbb {R}}$ 
&  14  
&    $16=16$                
\\
\hline
${\mathfrak{e}}_{6(-26)}$ 
&    2   
&  8   
&  24  
& ${\mathfrak{so}}(9,1)+{\mathbb{R}}$ 
&   8  
&     $16>8$             
\\
\hline
&        
&      
&      
& ${\mathfrak {sl}}(8,{\mathbb{R}})$ 
&  28 
&   $7< 35$              
\\
\cline{5-5}\cline{6-6}\cline{7-7}
${\mathfrak{e}}_{7(7)}$ 
&    7   
&  1   
&  63  
& ${\mathfrak{so}}(6,6)+{\mathfrak {sl}}(2,{\mathbb{R}})$ 
&  31  
&   $7< 32$                 
\\
\cline{5-5}\cline{6-6}\cline{7-7}  
&        
&      
&      
& ${\mathfrak{e}}_{6(6)}+{\mathbb{R}}$ 
&  36  
&    $7<27$               
\\
\hline
&        
&      
&      
& ${\mathfrak{so}}(8,4)+{\mathfrak {su}}(2)$ 
&  28  
&    $16<32$               
\\
\cline{5-5}\cline{6-6}\cline{7-7}
\raisebox{1.5ex}[0cm][0cm]{${\mathfrak{e}}_{7(-5)}$} 
&  \raisebox{1.5ex}[0cm][0cm]{4}   
&  \raisebox{1.5ex}[0cm][0cm]{4}   
&  \raisebox{1.5ex}[0cm][0cm]{60}        
& ${\mathfrak{so}}^{\ast}(12)+{\mathfrak {sl}}(2,{\mathbb{R}})$ 
&  28  
&    $16<32$                
\\
\hline
&        
&      
&      
& ${\mathfrak{e}}_{6(-26)}+{\mathbb{R}}$ 
&  24  
&    $24<27$     
\\
\cline{5-5}\cline{6-6}\cline{7-7}
\raisebox{1.5ex}[0cm][0cm]{${\mathfrak{e}}_{7(-25)}$} 
&    \raisebox{1.5ex}[0cm][0cm]{3}   
&  \raisebox{1.5ex}[0cm][0cm]{8}   
&  \raisebox{1.5ex}[0cm][0cm]{51}      
& ${\mathfrak{so}}(10,2)+{\mathfrak {sl}}(2,{\mathbb{R}})$ 
&  19  
&    $24<32$            
\\
\hline
&        
&      
&      
& ${\mathfrak{so}}(8,8)$ 
&  56  
&  $8<  64$               
\\
\cline{5-5}\cline{6-6}\cline{7-7}
\raisebox{1.5ex}[0cm][0cm]{${\mathfrak{e}}_{8(8)}$} 
&    \raisebox{1.5ex}[0cm][0cm]{8}   
&  \raisebox{1.5ex}[0cm][0cm]{1}   
& \raisebox{1.5ex}[0cm][0cm]{120}  
& ${\mathfrak{e}}_{7(7)}+{\mathfrak{sl}}(2,{\mathbb{R}})$ 
&  64  
&    $8<56$                
\\
\hline
&        
&      
&      
& ${\mathfrak{so}}(12,4)$ 
&  44  
&    $32<64$               
\\
\cline{5-5}\cline{6-6}\cline{7-7}
\raisebox{1.5ex}[0cm][0cm]{${\mathfrak{e}}_{8(-24)}$} 
&    \raisebox{1.5ex}[0cm][0cm]{4}   
&  \raisebox{1.5ex}[0cm][0cm]{8}   
& \raisebox{1.5ex}[0cm][0cm]{108}      
& ${\mathfrak{e}}_{7(-25)}+{\mathfrak{sl}}(2,{\mathbb{R}})$ 
&  52  
&    $32<56$               
\\
\hline
&        
&      
&      
& ${\mathfrak{so}}(5,4)$ 
&  16  
&    $4< 8$              
\\
\cline{5-5}\cline{6-6}\cline{7-7}
\raisebox{1.5ex}[0cm][0cm]{${\mathfrak{f}}_{4(4)}$} 
&    \raisebox{1.5ex}[0cm][0cm]{4}   
&  \raisebox{1.5ex}[0cm][0cm]{1}   
&  \raisebox{1.5ex}[0cm][0cm]{24}  
& ${\mathfrak{sp}}(3,{\mathbb{R}})+{\mathfrak{sl}}(2,{\mathbb{R}})$ 
&  10  
&   $4< 14$               
\\
\hline
${\mathfrak{f}}_{4(-20)}$ 
&    1   
&  8   
&  15  
& ${\mathfrak{so}}(8,1)$ 
&   7 
&     $8=8$              
\\
\hline
${\mathfrak{g}}_{2(2)}$ 
&    2   
&  1   
&   6  
& ${\mathfrak{sl}}(2,{\mathbb{R}})+{\mathfrak{sl}}(2,{\mathbb{R}})$ 
&   2  
&  $2<4$              
\\
\hline
\end{tabular}\\
\end{table}

\section{Applications to branching problems}
\label{sec:fm}
This section is devoted
 to applications
 of our classification results
 (Theorem \ref{thm:1.1} and Proposition \ref{prop:B})
 to branching problems
 of real reductive groups.  
Given an irreducible representation $\pi$ of $G$, 
 we wish to understand
 how the representation $\pi$ behaves
 as a representation of a subgroup $H$
 ({\it{branching problems}}).  
Basic quantities are the dimension 
 of continuous $H$-homomorphisms
\[
  m(\pi, \tau):=\dim \operatorname{Hom}_H(\pi|_H, \tau), 
\] 
 for irreducible representations $\tau$ of $H$.  
Concrete analysis 
 of the restriction $\pi|_H$
 could be reasonably developed
 under the condition
 that $m(\pi, \tau)< \infty$.  
However,
 the finiteness
 of the multiplicities 
 does not always hold
 even when $H$ is a maximal subgroup
 of $G$
 (see \cite{Kb2, xkeastwood60}
 for good behaviors
 and bad behaviors
 of the restriction $\pi|_H$).  
The initial motivation 
of our work 
 is to single out
 a nice framework 
 on the pair $(G,H)$ of reductive groups 
 for which we could expect
 that the branching laws
 $\pi|_H$ behave reasonably
 for {\it{arbitrary}} irreducible representations $\pi$.

\subsection{Admissible smooth representations}
\label{subsec:adm}
We begin with a quick review of some basic notion 
 of (infinite-dimensional) continuous representations
 of real reductive groups.

Suppose $G$ is a real reductive linear Lie group 
 (or its finite cover)
 and $K$ is a maximal compact subgroup.

Let $\pi$ be a continuous representation
 of $G$
 on a complete, locally convex vector space
 ${\mathcal{H}}$.  
The space ${\mathcal{H}}^{\infty}$ of 
 $C^{\infty}$-vectors of $(\pi,{\mathcal{H}})$
 is naturally endowed
 with Fr{\'e}chet topology,
 and we obtain a continuous representation 
 $\pi^{\infty}$ of $G$
 on ${\mathcal{H}}^{\infty}$.

Suppose 
 that $(\pi, {\mathcal{H}})$ is of finite length,
 in other words,
 suppose that there are only finitely many closed invariant subspaces
 in ${\mathcal{H}}$.  
We say $\pi$ is {\it{admissible}}
 (or $K$-{\it{admissible}})
 if 
\[
   \dim \operatorname{Hom}_K(\tau, \pi|_K)< \infty
\]
 for any irreducible finite-dimensional representation
 $\tau$ of $K$.  
For an admissible representation $(\pi, {\mathcal{H}})$
 such that ${\mathcal{H}}$ is a Banach space, 
 we say $(\pi^{\infty}, {\mathcal{H}}^{\infty})$
 is an {\it{admissible smooth representation}}.  
By the Casselman--Wallach globalization theory, 
there is a canonical equivalence of categories
 between the category of $({\mathfrak {g}}, K)$-modules
 of finite length 
 and the category of admissible smooth representations 
 of $G$.  
An admissible smooth representation is
 sometimes referred to as a smooth Fr{\'e}chet representation
 of moderate growth (\cite[Chapter 11]{WaI}).  
An irreducible admissible smooth representation
 of $G$ is said to be an {\it{irreducible smooth representation}}
 for simplicity 
 throughout this article.

\subsection{Finite multiplicity 
 property in branching laws}
\label{subsec:fm}

Suppose that $G$ is a real reductive linear Lie group
 and $H$ is a reductive subgroup
 defined algebraically 
 over ${\mathbb{R}}$.  
In what follows, 
the results remain true
 if we replace $(G,H)$ by their finite coverings
 or by their finite-index subgroups.  
Following the terminology in \cite{xtoshitoshima},
 we formulate a finite-multiplicity property
 on the pair $(G,H)$
 for the restriction of admissible representations:
\begin{enumerate}
\item[({\textbf{FM}})]
(Finite-multiplicity property)
\enspace
$\dim \operatorname{Hom}_H(\pi|_H, \tau)<\infty$, 
for any admissible smooth representation $\pi$ 
 of $G$
 and for any admissible smooth
representation $\tau$
 of $H$.     
\end{enumerate}

Here $\operatorname{Hom}_H(\, ,\,)$ denotes
 the space of continuous $H$-homomorphisms.  

As a direct consequence
 of Theorem \ref{thm:1.1} and Fact \ref{fact:1.4} , 
 we obtain a complete classification of the reductive symmetric pairs 
 $(G,H)$
 having the finite-multiplicity property {\rm{(FM)}}.  

\begin{thm}
\label{thm:fm}
Suppose $(G,H)$ is a reductive symmetric pair.  
Then the following two conditions
 are equivalent:
\begin{enumerate}
\item[{\rm{(i)}}]
$(G,H)$ satisfies the finite-multiplicity property {\rm{(FM)}}
 for the restriction
 of admissible smooth representations.  
\item[{\rm{(ii)}}]
The pair $({\mathfrak {g}}, {\mathfrak {h}})$ of Lie algebras
 is isomorphic to the direct sum
 of the pairs {\rm{(A)}}--{\rm{(H)}} in Theorem \ref{thm:1.1}
 up to outer automorphisms.  
\end{enumerate}
\end{thm}

\begin{remark}
\label{rem:fm}
{\rm{
Here are some features
 of the implication (ii) $\Rightarrow$ (i)
 in Theorem \ref{thm:fm}
 for the following special settings
 among (A)--(H):
\begin{enumerate}
\item[1)]
For the pairs (B) and (C), 
 the finite-multiplicity property (FM)
 is obvious 
 because $\pi$ is a finite-dimensional representation.  
\item[2)]
For the pairs (D)
 ({\it{i.e.}} $H=K$), 
 the finite-multiplicity property (FM)
 is trivial by the definition
 of admissible representations.  
(However,
 there are a number of equivalent conditions
 on admissibility,
 and the proof of Fact \ref{fact:1.4}
 given in \cite{xtoshitoshima}
 is not a tautology
 for $H=K$ but includes a microlocal proof
 of the classical fact 
that quasisimple irreducible representations
 are admissible,
 which was first proved by Harish-Chandra
 \cite{HC}.)
\item[3)]
For the pairs (F), 
 we have a uniform estimate of the multiplicities,
 as we shall see in Subsection \ref{subsec:BM}.  
\item[4)]
For the pairs (G), 
{\it{i.e.}}, 
 $(G,H)=(G' \times G', \operatorname{diag}G')$, 
 the finite-multiplicity property (FM)
 can be interpreted
 as the finiteness of linearly independent invariant trilinear forms, 
 see Subsection \ref{subsec:group}.
\end{enumerate}
}}
\end{remark}

\begin{remark}
\label{rem:fm2}
{\rm{
The property (FM) is a condition 
 on the pair $(G,H)$ 
 of groups 
 that assures the finiteness
 of the multiplicity $m(\pi, \tau)$
 for {\it{arbitrary}}
 $\pi$ and $\tau$.  
On the other hand,
 we may discuss 
 a condition 
 on the triple $(G,H, \pi)$
 for which $m(\pi, \tau)$ is finite
 for arbitrary $\tau$.  
This direction was 
 pursued in \cite{Kb2}
 under the additional assumption 
 of discrete decomposability 
 of branching laws
 (referred to as {\it{$H$-admissible restriction}}),
 and the classification theory 
 has been recently studied in 
 \cite{decoAq,xtoshiyoshima}, 
 particularly
 for \lq\lq{relatively small}\rq\rq\
infinite-dimensional representations
 $\pi$ of $G$
 ({\it{e.g.}},
 Zuckerman's derived functor modules,
 minimal representations, 
 {\it{etc.}}).  
}}
\end{remark}
\vskip 1pc
\subsection{Uniformly bounded multiplicities}
\label{subsec:BM}
In addition to the aforementioned finite-multiplicity property 
 {\bf{(FM)}}, 
 we consider the following two properties
 on a pair of reductive groups $(G,H)$:

\begin{enumerate}
\item[{{\bf{(BM)}}}]
({\it{Bounded-multiplicity restriction}})
\enspace
There exists a constant $C< \infty$
 such that 
\[
   \dim \operatorname{Hom}_H(\pi|_H, \tau)
\le C,
\]
for any irreducible smooth representation $\pi$ 
 of $G$
 and for any irreducible smooth representation $\tau$
 of $H$.
\item[{\bf{(MF)}}]
({\it{Multiplicity-free restriction}})\enspace
\[
   \dim \operatorname{Hom}_H  (\pi|_H, \tau) \le 1
\]
 for any irreducible smooth representation $\pi$ of $G$
 and for any irreducible smooth representation $\tau$ of $H$.  
\end{enumerate}
\vskip 1pc
Clearly, 
 we have
$
\text{(MF) $\Rightarrow$ (BM) $\Rightarrow$ (FM)}.  
$
Fact \ref{fact:1.4} is summarized by the following equivalences
 in the vertical direction:
\begin{alignat*}{6}
& \text{(MF)}
&& \quad\,\,\Rightarrow
&& \text{(BM)}
&& \Rightarrow
&& \text{(FM)}
&& \cdots \text{Representation Theory}
\\
&
&& {\text{\small{\cite[Theorem D]{xtoshitoshima}}}}
&& \,\,\Updownarrow
&& 
&& \,\,\Updownarrow {\text{\small{\cite[Theorem C]{xtoshitoshima}}}}
&& 
\\
&
&& 
&& \text{(BB)}
&& \Rightarrow
&& \text{(PP)}
&& \cdots \text{Geometry of flag varieties}
\\
\end{alignat*}

We note 
 that the properties (FM) and (BM)
 depend only on the Lie algebras
 $({\mathfrak {g}}, {\mathfrak {h}})$.  
Moreover,
 the bounded-multiplicity property (BM)
 depends only on the complexified Lie algebra
 $({\mathfrak {g}}_{\mathbb{C}}, {\mathfrak {h}}_{\mathbb{C}})$, 
 as was proved in \cite{xtoshitoshima}.  
On the other hand, 
 the multiplicity-free property (MF)
 is not determined
 by the pair of Lie algebras $({\mathfrak {g}}, {\mathfrak {h}})$,
 but depends on the groups $G$ and $H$.  
For example, 
 the best constant $C=2$
 if $(G,H)=(SL(2,{\mathbb{R}}), SO(1,1))$
 and $C=1$
 if $(G',H')=(O(2,1), O(1,1))$
 although the Lie algebras $({\mathfrak {g}}, {\mathfrak {h}})$
 and $({\mathfrak {g}}', {\mathfrak {h}}')$
 are isomorphic to each other.

As a corollary of Fact \ref{fact:1.4}
 and Proposition \ref{prop:cpx}, 
 we have a classification of symmetric pairs
 $({\mathfrak {g}}, {\mathfrak {h}})$
 satisfying the property (BM):
\begin{cor}
\label{cor:B}
Suppose $({\mathfrak {g}}, {\mathfrak {h}})$ is a 
 real reductive symmetric pair.  
Then the following three conditions
 are equivalent:
\begin{enumerate}
\item[{\rm{(i)}}]
For any real reductive Lie groups $G \supset H$
 with Lie algebras ${\mathfrak {g}} \supset {\mathfrak {h}}$, 
 respectively,
 the pair $(G,H)$ satisfies
 the bounded multiplicity property
 {\rm{(BM)}} for restriction.  

\item[{\rm{(ii)}}]
There exists a pair of 
 {\rm{(}}possibly disconnected{\rm{)}}
 real reductive Lie groups
 $G \supset H$
 such that 
 $(G,H)$ satisfies the multiplicity-free property
 {\rm{(MF)}} for restriction.  
\item[{\rm{(iii)}}]
The pair of the Lie algebras 
 $({\mathfrak{g}},{\mathfrak{h}})$
 is isomorphic 
 {\rm{(}}up to outer automorphisms{\rm{)}}
 to the direct sum of pairs {\rm{(A)}}, {\rm{(B)}}
 and {\rm{(F1)}} -- {\rm{(F5)}}.  
\end{enumerate}
\end{cor}
The implication (ii) $\Rightarrow$ (i) is obvious
 as mentioned.  
The equivalence (i) $\Leftrightarrow$ (iii)
is given in \cite[Theorem D]{xtoshitoshima}.  
The implication (iii) $\Rightarrow$ (ii) was proved 
 in Sun--Zhu \cite{SZ}.  
(Thus there are two different proofs
 for the implication (iii) $\Rightarrow$ (i).)
As a more refined form of the implication 
 (iii) $\Rightarrow$ (ii), 
Gross and Prasad \cite{GP}
 formulated a conjecture
 about the restriction
 of an irreducible admissible tempered representation
 of an inner form $G$
 of the group $O(n)$
 over a local field 
 to a subgroup which is an inner form $O(n-1)$
 ({\it{cf.}} \enspace
 (F2) and (F4) for the Archimedean field).  

\begin{exmp}
\label{ex:sbon}
{\rm{
For the pair $(G,H)
 =(O(n+1,1), O(n,1))$, 
 the space $\operatorname{Hom}_H(\pi|_H, \tau)$
 of continuous $H$-homomorphisms
 was classified in \cite{xtsbon}
 for all spherical principal series representations
 $\pi$ and $\tau$
 of $G$ and $H$, 
 respectively.  
This corresponds to a special case
 of (F5) in Corollary \ref{cor:B}.  
The classification was
 based on the explicit orbit decomposition
\cite[Chapter 5]{xtsbon}
\[
   G \backslash (G \times G)/(P_G \times P_G)
 \simeq 
   P_G \backslash G / P_G, 
\]
and a meromorphic family of $H$-intertwining operators
 were constructed for each orbit.  
}}
\end{exmp}

\subsection{Invariant trilinear forms}
\label{subsec:group}
A special case
 of a symmetric pair
 is the group case
\[
   (G,H)=(G' \times G', {\operatorname{diag}} G'), 
\]
 for which the branching problem deals
 with the decomposition of the tensor product
 of two irreducible representations of the group $G'$.  

Furthermore,
 the pair $(G' \times G', \operatorname{diag}G')$
 satisfies {\rm{(PP)}}
 if and only if 
 the homogeneous space
 $(G' \times G' \times G')/\operatorname{diag} G'$
 is a real spherical variety
 in view of the following isomorphism:
\begin{equation*}
(P_{G'}\times P_{G'}\times P_{G'})\backslash(G' \times G' \times G')/
\operatorname{diag}G'
\simeq 
(P_{G'} \times P_{G'})\backslash(G' \times G')/ P_{G'}.   
\end{equation*}

By these observations,
 we can interpret Theorem \ref{thm:fm}
 and Corollary \ref{cor:B}
 in the following form
 (cf. \cite{xtoshi95}):

\begin{cor}
\label{cor:1.2-copy}
Suppose $G$ is a simple Lie group.  
Then the following three conditions on $G$ are equivalent:
\begin{enumerate}
\item[{\rm{(i)}}]
For any triple of admissible smooth representations $\pi_1$, $\pi_2$, 
 and $\pi_3$ of $G$, 
\[
\dim \operatorname{Hom}_G(\pi_1 \otimes \pi_2, 
\pi_3)< \infty.  
\]
\item[{\rm{(ii)}}]
For any triple of admissible smooth representations
 $\pi_1$, $\pi_2$ and $\pi_3$ of $G$, 
 the space of invariant trilinear forms
 is finite-dimensional:
\[
   \dim \operatorname{Hom}_G(\pi_1 \otimes \pi_2 \otimes \pi_3, {\mathbb{C}})<\infty.  
\]
\item[{\rm{(iii)}}]
Either $G$ is compact
 or ${\mathfrak {g}}$ is isomorphic to ${\mathfrak{o}}(n,1)$
 $(n \ge 2)$.  
\end{enumerate}
\end{cor}
\vskip 1pc
\begin{cor}
Suppose $G$ is a simple Lie group.  
Then the following three conditions on $G$
 are equivalent:
\begin{enumerate}
\item[{\rm{(i)}}]
There exists a constant $C< \infty$
such that 
\[
\dim \operatorname{Hom}_G(\pi_1 \otimes \pi_2, \pi_3) \le C,   
\]
 for any irreducible smooth representations
 $\pi_1$, $\pi_2$, and 
 $\pi_3$ of $G$.  
\item[{\rm{(ii)}}]
There exists a constant $C< \infty$
such that 
\[
\dim \operatorname{Hom}_G(\pi_1 \otimes \pi_2 \otimes \pi_3, {\mathbb{C}})
\le C,   
\]
 for any irreducible smooth representations
 $\pi_1$, $\pi_2$, and 
 $\pi_3$ of $G$.  
\item[{\rm{(iii)}}]
The Lie algebra ${\mathfrak {g}}$ is isomorphic
 to one of 
$
  {\mathfrak {su}}(2) \simeq {\mathfrak {o}}(3)
$, 
$
   {\mathfrak {su}}(1,1) \simeq {\mathfrak {sl}}(2,{\mathbb{R}})
   \simeq 
   {\mathfrak {o}}(2,1)
$
 or 
$
   {\mathfrak {sl}}(2,{\mathbb{C}})
   \simeq
   {\mathfrak {o}}(3,1)
$.  
\end{enumerate}
\end{cor}
Built on the nice properties
 in Corollary \ref{cor:1.2-copy}, 
 a meromorphic family 
 of invariant trilinear forms
 of principal series representations
 of the Lorentz group $O(n,1)$
 was studied in \cite{CKOP}.

\end{document}